\documentclass[12pt]{amsart}

\usepackage{amsfonts}
\usepackage{amsmath}
\usepackage{amsrefs}
\usepackage{amssymb}
\usepackage{hyperref}
\usepackage{url}
\usepackage[shortlabels]{enumitem}
\usepackage[paper=a4paper,left=20mm,right=20mm,top=25mm,bottom=25mm]{geometry}

\usepackage{tikz}
\usetikzlibrary{backgrounds}
\usetikzlibrary{trees,snakes}
\usepackage[outline]{contour}
\contourlength{1.5pt}
\usetikzlibrary{positioning}
\usetikzlibrary{%
  matrix,%
  calc,%
  arrows%
}

\let\emptyset\varnothing
\nocite{*}

\newcommand{\NN}{\mathbb{N}}

\newcommand{\ZZ}{\mathbb{Z}}

\newcommand{\PP}{\mathbb{P}}

\newcommand{\cC}{\mathcal{C}}
\newcommand{\cD}{\mathcal{D}}

\newcommand{\cE}{\mathcal{E}}
\newcommand{\cF}{\mathcal{F}}

\newcommand{\cK}{\mathcal{K}}

\newcommand{\cT}{\mathcal{T}}

\newcommand{\cW}{\mathcal{W}}
\newcommand{\cX}{\mathcal{X}}

\newcommand{\cI}{\mathcal{I}}
\newcommand{\cP}{\mathcal{P}}

\newcommand{\cR}{\mathcal{R}}

\DeclareMathOperator{\fF}{\mathfrak{F}}

\DeclareMathOperator{\res}{res}
\DeclareMathOperator{\Rad}{Rad}
\DeclareMathOperator{\rad}{rad}
\DeclareMathOperator{\rep}{rep}
\DeclareMathOperator{\ind}{ind}
\DeclareMathOperator{\add}{add}

\DeclareMathOperator{\Char}{char}
\DeclareMathOperator{\Hom}{Hom}
\DeclareMathOperator{\coker}{coker}

\DeclareMathOperator{\Gr}{Gr}
\DeclareMathOperator{\modd}{mod}

\DeclareMathOperator{\Ext}{Ext}
\DeclareMathOperator{\Soc}{Soc}

\DeclareMathOperator{\im}{im}
\DeclareMathOperator{\End}{End}

\DeclareMathOperator{\GL}{GL}

\DeclareMathOperator{\EIP}{EIP}
\DeclareMathOperator{\EKP}{EKP}
\DeclareMathOperator{\ESP}{ESP}
\DeclareMathOperator{\ERP}{ERP}
\DeclareMathOperator{\CSR}{CSR}
\DeclareMathOperator{\CRR}{CRR}
\DeclareMathOperator{\CR}{CR}
\DeclareMathOperator{\dimu}{\underline{dim}}

\let\emptyset\varnothing
\newtheorem{proposition}{Proposition}[subsection]
\newtheorem{Theorem}[proposition]{Theorem}
\newtheorem{Lemma}[proposition]{Lemma}

\newtheorem{corollary}[proposition]{Corollary}
\newtheorem*{TheoremN}{Theorem}
\newtheorem*{corollaryN}{Corollary}
\newtheorem{TheoremS}{Theorem}[section]
\newtheorem{corollaryS}[TheoremS]{Corollary}
\newenvironment{example}[1][Example.]{\begin{trivlist}
\item[\hskip \labelsep {\bfseries #1}]}{\end{trivlist}}
\newenvironment{examples}[1][Examples.]{\begin{trivlist}
\item[\hskip \labelsep {\bfseries #1}]}{\end{trivlist}}
\newenvironment{Remark}[1][Remark.]{\begin{trivlist}
\item[\hskip \labelsep {\bfseries #1}]}{\end{trivlist}}
\newenvironment{Remarks}[1][Remarks.]{\begin{trivlist}
\item[\hskip \labelsep {\bfseries #1}]}{\end{trivlist}}
\newenvironment{Definition}[1][Definition.]{\begin{trivlist}
\item[\hskip \labelsep {\bfseries #1}]}{\end{trivlist}}

\begin{document}

\title{Representations of constant socle rank for the Kronecker algebra}
\author{Daniel Bissinger}
\address{Christian-Albrechts-Universit\"at zu Kiel, Ludewig-Meyn-Str. 4, 24098 Kiel, Germany}
\email{bissinger@math.uni-kiel.de}
\date{}
\thanks{Partly supported by the D.F.G. priority program SPP 1388  ``Darstellungstheorie''}
\maketitle

\begin{abstract} 
Inspired by recent work of Carlson, Friedlander and Pevtsova concerning modules for $p$-elementary abelian groups $E_r$ of rank $r$ over a field of characteristic $p > 0$, we introduce the notions of modules with constant $d$-radical rank and modules with constant $d$-socle rank for the generalized Kronecker algebra $\cK_r = k\Gamma_r$ with $r \geq 2$ arrows and $1 \leq d \leq r-1$.\\
We study subcategories given by modules with the equal $d$-radical property and the equal $d$-socle property. Utilizing the Simplification method due to Ringel, we prove that these subcategories in $\modd \cK_r$ are of wild type. Then we use a natural functor $\fF \colon \modd \cK_r \to \modd kE_r$ to transfer our results to $\modd kE_r$.
 \end{abstract}

\thispagestyle{empty} 

\section*{Introduction}

Let $r \geq 2$, $k$ be an algebraically closed field of characteristic $p > 0$ and $E_r$ be a $p$-elementary abelian group of rank $r$. It is well known that the category of finite-dimensional $kE_r$-modules $\modd kE_r$ is of wild type, whenever $p \geq 3$ or $p = 2$ and $r > 2$. Therefore subclasses with more restrictive properties have been studied; in  \cite{CFS1}, the subclass of modules of constant rank $\CR(E_r)$ and modules with even more restrictive properties, called equal images property and equal kernels property, were introduced.\\
In \cite{Wor1} the author defined analogous categories $\CR$, $\EIP$ and $\EKP$ in the context of the generalized Kronecker algebra $\cK_r$, and in more generality for the generalized Beilison algebra $B(n,r)$. Using a natural functor $\fF \colon \modd \cK_r \to \modd kE_r$ with nice properties, she gave new insights into the categories of equal images and equal kernels modules for $\modd kE_r$ of Loewy length $\leq 2$. A crucial step is the description of $\CR$, $\EIP$ and $\EKP$ in homological terms, involving a family $\PP^{r-1}$-family of regular "test"-modules.\\
Building on this approach, we show that the recently introduced modules \cite{CFP1} of constant socle rank and constant radical rank can be described in the same fashion. For $1 \leq d < r$ we introduce modules of constant $d$-radical rank $\CRR_d$ and constant $d$-socle rank $\CSR_d$ in $\modd \cK_r$. More restrictive - and also easier to handle - are modules with the equal $d$-radical property $\ERP_d$ and the equal $d$-socle property $\ESP_d$. For $d = 1$ we have $\ESP_1 = \EKP$, $\ERP_1 = \EIP$ and $\CSR_1 = \CR = \CRR_1$. Studying these classes in the hereditary module category $\modd \cK_r$ allows us to use tools not available in $\modd kE_r$. \\
As a first step, we establish a homological characterization of  $\CSR_d,\CRR_d,\ESP_d$ and $\ERP_d$. We denote by $\Gr_{d,r}$ the Grassmanian of $d$-dimensional subspaces of $k^r$. In generalization of \cite{Wor1}, we define a $\Gr_{d,r}$-family of  "test"-modules $(X_U)_{U \in \Gr_{d,r}}$ and show that this family can be described in purely  combinatorial terms by being indecomposable of dimension
vector $(1,r-d)$. This allows us to construct many examples of modules of equal socle rank in $\modd \cK_s$ for $s \geq 3$ by considering pullbacks along natural embeddings $\cK_r \to \cK_s$.

\indent Since $\cK_r$ is a wild algebra for $r > 2$ and every regular component in the Auslander-Reiten quiver of $\cK_r$ is of type $\ZZ A_\infty$, it is desirable to find invariants that give more specific information about the regular components. It is shown in \cite{Wor1} that there are uniquely determined quasi-simple modules $M_\cC$ and $W_\cC$ in $\cC$ such that
the cone $(M_\cC \rightarrow)$ of all successors of $M_\cC$ satisfies $(M_\cC \rightarrow) = \EKP \cap \cC$ and the cone $(\rightarrow W_\cC)$ of all predecessors of $W_\cC$ satisfies
$(\rightarrow W_\cC)= \EIP \cap \cC$. 
Using results on elementary modules, we generalize this statement for $\ESP_d$ and $\ERP_d$. Our main results may be summarized as follows:
\begin{TheoremN} Let $r \geq 3$ and $\cC$  be  a regular component of the Auslander-Reiten quiver of $\cK_r$.
\begin{enumerate}[topsep=0em, itemsep= -0em, parsep = 0 em, label=$(\alph*)$]
\item For each $1 \leq i < r$ the category $\Delta_i := \ESP_i \setminus \ESP_{i-1}$ is wild, where $\ESP_0 := \emptyset$.
\item For each $1 \leq i < r$ there exists a unique quasi-simple module $M_i$ in $\cC$ such that $\ESP_i \cap \cC = (M_i \to)$.
\item There exists at most one number  $1 < m(\cC) < r$ such that $\Delta_{m(\cC)} \cap \cC$ is non-empty. If such a number exists $\Delta_{m(\cC)} \cap \cC$ is the ray starting in $M_{m(\cC)}$. 
\end{enumerate} 
\end{TheoremN}

If there is no such number as in $(c)$ we set $m(\cC) = 1$.
An immediate consequence of $(b)$ and $(c)$ is that for $1 \leq i \leq j < r$ we have $(M_i \to) = (M_j \to)$ or $\tau M_i = M_j$. Moreover statement $(a)$ shows the existence of a lot of AR-components such that $\Delta_i \cap \cC$ is a ray and for every such component we have $m(\cC) = i$. With the dual result for $\ERP$ we assign a number $1\leq  w(\cC) < r$ to each regular component $\cC$, giving us  the possibility to distinguish $(r-1)^2$ different types of regular components.\\

\begin{figure*}[!h]
\centering 
\includegraphics[width=0.7\textwidth, height=175px]{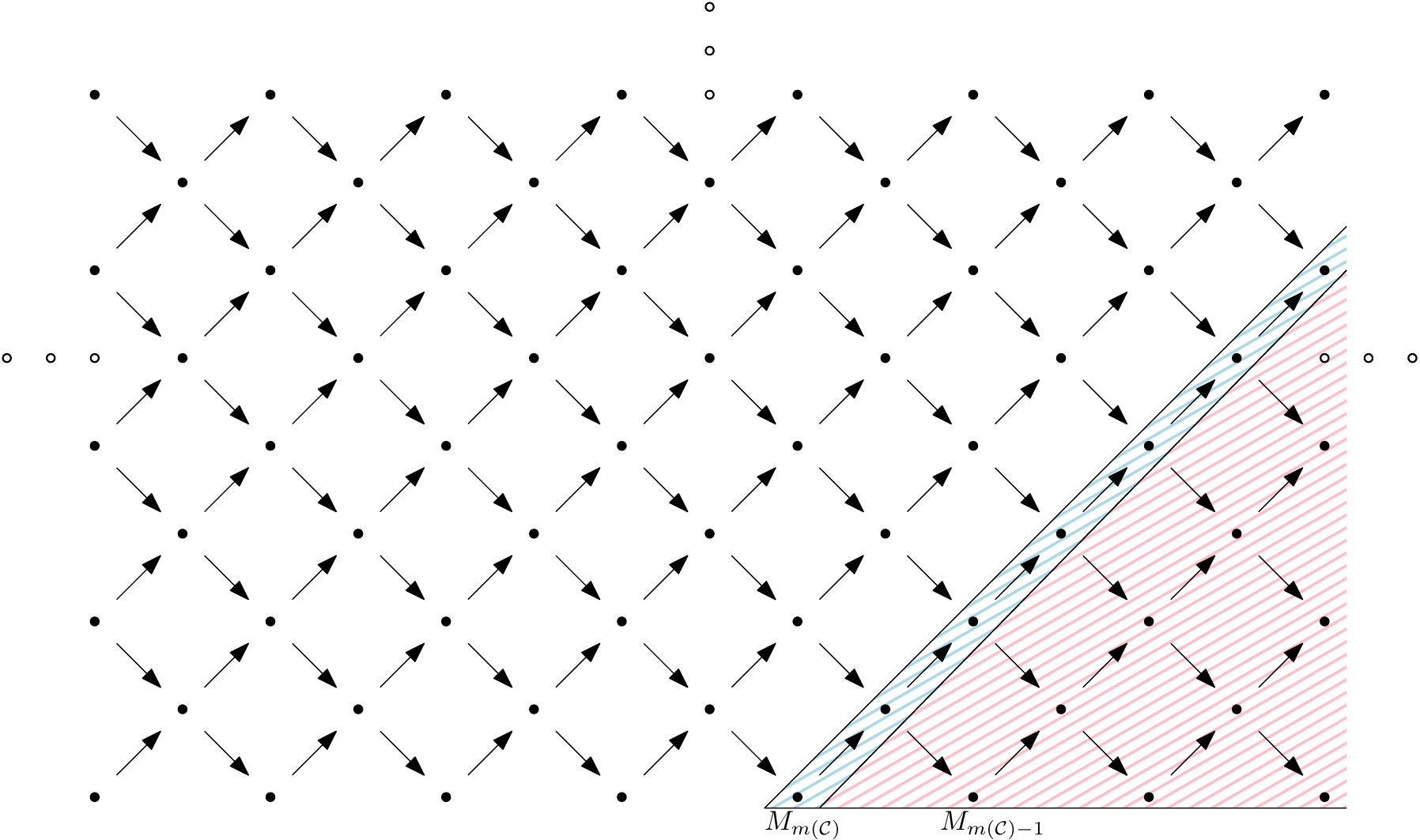}
\caption{Illustration of a regular component with $m(\cC) \neq 1$.}
\label{Fig:ZA_inf_cones}
\end{figure*}

 To prove statement $(a)$ we exploit the fact that every regular module $M$ in $\modd \cK_r$ has self-extensions with $\dim_k \Ext(M,M) \geq 2$ in conjunction with the Process of Simplification. This method was introduced in \cite{Ri4} and  produces extension closed subcategories, whose objects may be filtered by pairwise orthogonal bricks. For elementary abelian groups, $\modd kE_r$ is the only such subcategory. We therefore use the functor $\fF \colon \modd \cK_r \to \modd kE_r$, whose essential image consists of all modules of Loewy length $\leq 2$, to transfer our results to $\modd kE_r$. We denote with $\ESP_{2,d}(E_r)$ the category of modules in $\modd kE_r$ of Loewy length $\leq 2$ with the equal $d$-socle property and arrive at:

\begin{corollaryN}
Let $\Char(k) > 0$, $r \geq 3$ and $1 \leq d < r$. Then $\ESP_{2,d}(E_r) \setminus \ESP_{2,d-1}(E_r)$ has wild representation type, where $\ESP_{2,0}(E_r) := \emptyset$.
\end{corollaryN}

For $r = 2$, we consider the Beilinson algebra $B(3,2)$. The fact that $B(3,2)$ is a concealed algebra of type $Q = 1 \rightarrow 2 \rightrightarrows 3$ allows us to apply the Simplification process in $\modd kQ$. We find a wild subcategory in the category of all modules in $\modd B(3,2)$ with the equal kernels property and conclude:

\begin{corollaryN}
Assume that $\Char(k) = p > 2$, then the full subcategory of modules with the equal kernels property in $\modd kE_2$ and Loewy length $\leq 3$ has wild representation type.
\end{corollaryN}

\noindent In particular, we generalize results by Benson \cite{Be1}, and Bondarenko and Lytvynchuk \cite{Bo1} concerning the wildness of various subcategories of $kE_r$-modules. We also construct examples of regular components $\cC$ such that each module in $\cC$ has constant $d$-socle rank, but no module in $\cC$ is $\GL_r$-stable in the sense of \cite{CFP1}. 

\section{Preliminaries}

Throughout this article let $k$ be an algebraically closed field and $r \geq 2$. If not stated otherwise, $k$ is of arbitrary characteristic. $Q = (Q_0,Q_1)$ always denotes an acyclic, finite and connected quiver. For an arrow $\alpha \colon x \to y \in Q_1$ we define $s(\alpha) = x$ and $t(\alpha) = y$. We say that $\alpha$ starts in $s(\alpha)$ and ends in $t(\alpha)$. 
A finite dimensional representation $M = ((M_x)_{x \in Q_0},(M(\alpha))_{\alpha \in Q_1})$  over $Q$ consists of vector spaces $M_x$ and linear maps $M(\alpha) \colon M_{s(\alpha)} \to M_{t(\alpha)}$ such that $\dim_k M := \sum_{x \in Q_0} \dim_k M_x$ is finite. A morphism $f \colon M \to N$ between representations is a collection of linear maps $(f_z)_{z \in Q_0}$ such for each arrow $\alpha \colon x \to y$ there is a commutative diagram 

\[	   \begin{tikzpicture}[descr/.style={fill=white,inner sep=1.5pt}]
		
				\matrix [
            matrix of math nodes,
            row sep=3em,
            column sep=3.0em,
            text height=2.0ex, text depth=0.25ex
        ] (s)
{
& |[name=A_1]| M_x  &|[name=B_1]| M_y  \\
& |[name=A_2]| N_x &|[name=B_2]| N_y.  \\
};
\draw[->] (A_1) edge node[auto] {$M(\alpha)$} (B_1)         			 	  (A_2) edge node[auto] {$N(\alpha)$} (B_2)   
		  (A_1) edge node[auto] {$f_x$} (A_2)	
		  (B_1) edge node[auto] {$f_y$} (B_2)					      
				  ;
\end{tikzpicture}\]

\noindent The category of finite dimensional representations over $Q$ is denoted by $\rep(Q)$ and $kQ$ is the path algebra of $Q$ with idempotents $e_x, x \in Q_0$. $kQ$ is a finite dimensional, associative, basic and connected $k$-algebra. Let $\modd kQ$ be the class of finite-dimensional $kQ$ left modules. Given $M \in \modd kQ$ we set $M_x := e_x M$. The categories $\modd kQ$ and $\rep(Q)$ are equivalent $($see for example \cite[III 1.6]{Assem1}$)$. We will therefore switch freely between representations of $Q$ and modules of $kQ$, if one of the approaches seems more convenient for us. We assume that the reader is familiar with Auslander-Reiten theory and basic results on wild hereditary algebras. For a well written survey on the subjects we refer to \cite{Assem1}, \cite{Ker2} and \cite{Ker3}.

\begin{Definition}
 Recall the definition of the dimension function
\[ \dimu \colon \modd kQ \to \ZZ_{0}^{Q_0}, M \mapsto (\dim_k M_x)_{x \in Q_0}.\] 
If $0 \to A \to B \to C \to 0$ is an exact sequence, then $\dimu A + \dimu C = \dimu B$. The quiver $Q$ defines a $($non-symmetric) bilinear form 
\[ \langle -,- \rangle \colon \ZZ^{Q_0} \times \ZZ^{Q_0} \to \ZZ, \]
given by $((x_i),(y_j)) \mapsto \sum_{i \in Q_0} x_i y_i - \sum_{\alpha \in Q_1} x_{s(\alpha)} y_{t(\alpha)}$. For the case that $x,y$ are given by dimension vectors of modules there is another description of $\langle -,- \rangle$ known as the Euler-Ringel form \cite{Ri4}:

\[ \langle \dimu M, \dimu N \rangle = \dim_k \Hom(M,N) - \dim_k \Ext(M,N).\]

\noindent We denote with $q = q_{Q} \colon \ZZ^{Q_0} \to \ZZ$ the corresponding quadratic form. A vector $d \in \ZZ^{Q_0}$ is called a real root if $q(d) = 1$ and an imaginary root if $q(d) \leq 0$.
\end{Definition}

\noindent Denote by $\Gamma_r$ the $r$-Kronecker quiver, which is given by two vertices $1,2$ and arrows $\gamma_1,\ldots,\gamma_r \colon 1 \to 2$.

\begin{figure}[!h]
\centering 
\tikzstyle{every node}=[]
\tikzstyle{edge from child}=[]

\begin{tikzpicture}[->,>=stealth',auto,node distance=3cm,
  thick,main node/.style={circle,draw,font=\sffamily\Large\bfseries}]

  \node (1) {$\bullet$};
  \node (2) [right of=1] {$\bullet$};
   \node[color=black] at (1.4,0.7) {$\gamma_{1}$};
   \node[color=black] at (1.4,-0.7) {$\gamma_{r}$};
   \node[color=black] at (1.4,-0.2) {$\vdots$};
   \node[color=black] at (1.4,0.4) {$\vdots$};
   
   \node[color=black] at (0,-0.3) {$1$};
   \node[color=black] at (3,-0.3) {$2$};

  \path[every node/.style={font=\sffamily\small}]
    
    (1) edge[bend right] node [left] {} (2)
    (1) edge[bend left] node [left] {} (2)
    (1) edge[] node [left] {} (2)
   ;
\end{tikzpicture}
\caption{The Kronecker quiver $\Gamma_r$.}
\label{Fig:Kronecker}

\end{figure}
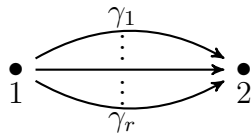

We set $\cK_r := k\Gamma_r$ and $P_1 := \cK_r e_2$, $P_2 := \cK_r e_1$. $P_1$ and $P_2$ are the indecomposable projective modules of $\modd \cK_r$, $\dim_k \Hom(P_1,P_2) = r$ and $\dim_k \Hom(P_2,P_1) = 0$. As Figure $\ref{Fig:Kronecker}$ suggests, we write $\dimu M = (\dim_k M_1,\dim_k M_2)$. For example $\dimu P_1 = (0,1)$ and $\dim_k P_2 = (1,r)$.
The Coxeter matrix $\Phi$ and its inverse $\Phi^{-1}$ $($see for example \cite[3.1]{Ker3}$)$ are
$\Phi:=\begin{pmatrix}
 r^2-1 & -r \\ 
 r & -1 \\
\end{pmatrix}$,
$\Phi^{-1}=\begin{pmatrix}
 -1 & r \\ 
 -r & r^2-1 \\
\end{pmatrix}$. 
For $M$ indecomposable it is $\dimu \tau M = \Phi(\dimu M)$ if $M$ is not projective and 
$\dimu \tau^{-1} M = \Phi^{-1}(\dimu M)$ if $M$ is not injective. The quadratic form $q$ is given by $q(x,y) = x^2 + y^2 - r x y$.

\begin{figure}[!h]
\centering 
\includegraphics[width=0.7\textwidth, height=50px]{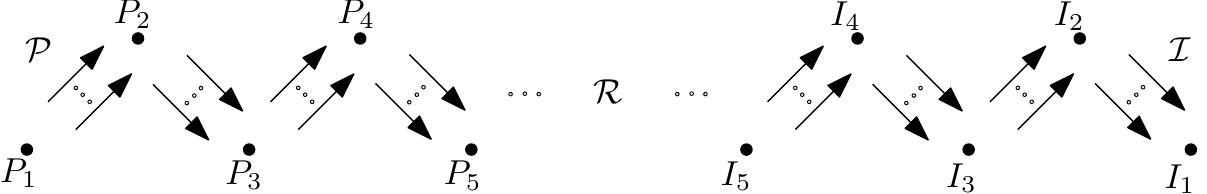}
\caption{AR quiver of $\cK_r$}
\label{Fig:AuslanderReitenquiver}
\end{figure}

Figure $\ref{Fig:AuslanderReitenquiver}$  shows the notation we use for the components $\cP,\cI$ in the Auslander-Reiten quiver of $\cK_r$ which contain the indecomposable projective modules $P_1,P_2$ and indecomposable injective modules $I_1,I_2$. The set of all other components is denoted by $\cR$.

Ringel has proven \cite[2.3]{Ri3} that every component in $\cR$ is of type $\ZZ A_\infty$ if $r \geq 3$ or a homogenous tube $\ZZ A_\infty/\langle \tau \rangle$ if $r = 2$. A module in such a component is called regular. An irreducible morphism in a regular component is injective if the corresponding arrow is uprising $($see Figure $\ref{Fig:ZA_inf_cones}$ for $r \geq 3)$ and surjective otherwise. A regular module $M$ is called quasi-simple, if the AR-sequence terminating in M has an indecomposable middle term. These modules are the modules in the bottom layer of the regular components. If $M$ is quasi-simple in a regular component $\cC$, there is an infinite chain $($ray$)$ of irreducible monomorphisms $($resp. epimorphisms$)$
\[ M = M[1] \to M[2] \to M[3] \to \cdots \to M[l] \to \cdots \] 
\[ \cdots (l)M \to \cdots \to (3)M \to (2)M \to (1)M = M, \] 
and for each regular module $X$ there are unique quasi-simple modules $N,M$ and $l \in \NN$ with $(l)M = X = N[l]$. $l$ is called the quasi-length of $X$.\\  
The indecomposable modules in $\cP$ are called preprojective modules and the modules in $\cI$ are called preinjective modules. Moreover we call an arbitrary module preprojective $($resp. preinjective, regular$)$ if all its indecomposable direct summands are preprojective $($res. preinjective, regular$)$. It is $P$ in $\cP$ $(I$ in $\cI)$ if and only if there is $l \in \NN_0$ with $\tau^l P = P_i$ $ (\tau^{-l} I = I_i)$ for $i \in \{1,2\}$. Recall that there are no homomorphisms from right to left \cite[VIII, 2.13]{Assem1}. To emphasize this result later on, we just write
 \[ \Hom(\cI,\cP) = 0 = \Hom(\cI,\cR) = 0 = \Hom(\cR,\cP).\]

Using the equivalence of categories $\modd \cK_r \cong \rep(Q)$ we introduce the duality $D \colon \modd \cK_r \to \modd \cK_r$ by setting $(DM)_{x} := (M_{\psi(x)})^\ast$ and $(DM)(\gamma_i) := (M(\gamma_i))^\ast$ for each $M \in \rep(Q)$, where $\psi \colon \{1,2\} \to \{1,2\}$  is the permutation of order $2$. Note that $D(P_i) = I_i$ for all $i \in \NN$. We state a simplified version of Kacs Theorem \cite[1.10]{Kac} for the Kronecker algebra in combination with results on the quadratic form proven in \cite[2.3]{Ri4}:

\begin{TheoremS} Let $r \geq 2$ and $d \in \NN^2_0$.
\begin{enumerate}[topsep=0em, itemsep= -0em, parsep = 0 em, label=$(\alph*)$]
\item If $d = \dimu M$ for some indecomposable module $M$, then $q(d)\leq 1$.
\item If  $q(d) = 1$, then there exists a unique indecomposable module $X$ with $\dimu X = d$. In this case $X$ is preprojective or preinjective and $X$ is preprojective if and only if $\dim_k X_1 < \dim_k X_2$.
\item If $q(d) \leq 0$, then there exist infinitely many indecomposable modules $Y$ with $\dimu Y = d$ and each $Y$ is regular.
\end{enumerate}
\end{TheoremS}

\noindent Since there is no pair $(a,b) \in \NN^2_0 \setminus \{(0,0)\}$ satisfying $a^2 + b^2 - rab = q(a,b) = 0$ for $r \geq 3$, we conclude together with \cite[VIII 2.7]{Assem1}:
\begin{corollaryS}\label{Extgeq2}
Let $M$ be an indecomposable $\cK_r$-module. Then $\Ext(M,M) = 0$ if and only if $M$ is preprojective or preinjective. If $r \geq 3$ and $M$ is regular, then $\dim_k \Ext(M,M) \geq 2$.
\end{corollaryS}

Let $\modd_{pf} \cK_r$ be the subcategory of all modules without non-zero projective direct summands and $\modd_{if} \cK_r$ the subcategory of all modules without non-zero injective summands. Since $\cK_r$ is a hereditary algebra, the Auslander-Reiten translation $\tau \colon \modd \cK_r \to \modd \cK_r$ induces an equivalence from $\modd_{pf}$ to $\modd_{if}$. In particular if $M$ and $N$ are indecomposable with $M,N$ not projective we get $\Hom(M,N) \cong \Hom(\tau M,\tau N)$. The Auslander-Reiten formula \cite[2.13]{Assem1} simplifies to 

\begin{TheoremS}\cite[2.3]{Ker3}
   For $X,Y$ in $\modd \cK_r$ there a functorial isomorphisms 
    \[ \Ext(X,Y) \cong D\Hom(Y,\tau X) \cong D\Hom(\tau^{-1} Y,X).\]
\end{TheoremS}

\section{Modules of constant radical and socle rank}

\subsection{Elementary modules of small dimension}
Let $r \geq 3$. The homological characterization in \cite{Wor1} uses a family of modules    of projective dimension $1$. In case $B(2,r) = \cK_r$ this condition is satisfied for every non-projective module and the familiy is denoted by $(X_\alpha)_{\alpha \in k^r\setminus \{0\}}$. We will see later on that each non-zero proper submodule of $X_\alpha$ is isomorphic to a finite number of copies of $P_1$ and $X_\alpha$ itself is regular. In particular we do not find a short exact sequence $0 \to A \to X_\alpha \to B \to 0$ such that $A$ and $B$ are regular and non-zero. In the language of wild hereditary algebras, we therefore deal with elementary modules:
\begin{Definition} \cite[1]{KerLuk1}
A non-zero regular module $E$ is called elementary, if there is no short exact sequence $0 \to A \to E \to B \to 0$ with $A$ and $B$ regular non-zero.
\end{Definition}

Elementary modules are analogues of quasi-simple modules in the tame hereditary case $(r = 2)$. If $X$ is a regular module, then $X$ has a filtration $0 = X_0 \subset X_1 \subset \cdots \subset X_r = X$ such that $X_i/X_{i-1}$ is elementary for all $1 \leq i \leq r$ and the elementary modules are the smallest class with that property. For basic results on elementary modules, used in this section, we refer to \cite{KerLuk1}.\\

\vspace{0.25cm}

\noindent We are grateful to Otto Kerner for pointing out the following helpful lemma.

\begin{Lemma}\label{CompElementary}
Let $E$ be an elementary module and $X,Y$ regular with non-zero morphisms $f \colon X \to E$ and $g \colon E \to Y$. Then $g \circ f \neq 0$. In particular $\End(E) = k$. 
\end{Lemma}
\begin{proof}
Since $f$ is non-zero and $\Hom(\cR,\cP) = 0 = \Hom(\cI,\cR)$,  $\im f$ is a regular non-zero submodule of  $E$. Consequently, since $E$
is elementary, $\coker f$ is preinjective by \cite[1.3]{KerLuk1}, hence $g$ can not factor through coker $f$.
\end{proof}

We use the theory on elementary modules to generalize \cite[2.7]{Wor1} in the hereditary case. 
\begin{proposition}\label{Cone1}
Let $\cE$ be a family of elementary modules of bounded dimension and put  
$ \cT(\cE) := \{ M \in \modd \cK_r \mid  \forall E \in \cE: \ \Ext(E,M) = 0 \}$. Then the following statements hold:
\begin{enumerate}[topsep=0em, itemsep= -0em, parsep = 0 em, label=$(\arabic*)$]
\item $\cT(\cC)$ is closed under extensions, images and $\tau$.
\item $\cT(\cC)$ contains all preinjective modules.
\item For each regular component $\cC$, the set $\cT(\cE) \cap \cC$ forms a non-empty cone in $\cC$, which consists of the predecessors of a uniquely determined quasi-simple module in $\cC$.
\end{enumerate}

\end{proposition}
\begin{proof}
Since $\Ext^2 = 0$, the category is closed under images. Let $M$ be indecomposable with $\Ext(E,M) = 0$ for all $E \in \cE$. Then the Auslander-Reiten formula yields $0 = \dim_k \Hom(M,\tau E)$ for all $E \in \cE$. Assume that $P = M$ is preprojective,  Let $l \in \NN_0$ such that $\tau^l P$ is projective, then $\tau^l P = P_i$ for an $i \in \{1,2\}$ and $0 = \dim_k \Hom(P,\tau E) = \dim_k \Hom(\tau^l P,\tau^{l+1} E) = (\dimu (\tau^{l+1} E))_{3-i}$. This is a contradiction, since every non-sincere $\cK_r$-module is semi-simple. Hence $M$ is regular or preinjective.\\ 
We are going to show that $\Ext(E,\tau M) = 0$. In view of the Auslander-Reiten formula we get $\dim_k \Ext(E,\tau M) = \dim_k \Hom(M,E)$. Let $f \colon M \to E$ be a morphism and assume that $f \neq 0$. Since $0 = \Hom(\cI,\cR)$ the module $M$ is not preinjective and therefore regular. As a regular module $E$ has self-extensions $($see \ref{Extgeq2}$)$ and hence $E \notin \cT(\cE)$. Therefore $0 \neq \dim_k \Ext(E,E) = \dim_k \Hom(E,\tau E)$ and we find $0 \neq g \in \Hom(E,\tau E)$. Lemma $\ref{CompElementary}$ provides a non-zero morphism
\[ M \stackrel{f}{\to} E \stackrel{g}{\to} \tau E.\]
We conclude $0 \neq \dim_k \Hom(M,\tau E) = \dim_k \Ext(E,M) = 0$, a contradiction. 
Now note that $I_1,I_2$ are contained in $\cT(\cE)$. The existence of the cones can be shown as in \cite[3.3]{Wor1} using the upper bound and \cite[10.7]{Ker3}.
\end{proof}

\noindent The next result follows by the Auslander-Reiten formula and duality since $D(E)$ is elementary if and only if $E$ is elementary.
\begin{proposition}\label{Cone2}
Let $\cE$ be a family of elementary modules of bounded dimension and 
put $ \cF(\cE) := \{ M \in \modd \cK_r \mid \forall E \in \cE: \ \Hom(E,M) = 0\}$. Then the following statements hold.
\begin{enumerate}[topsep=0em, itemsep= -0em, parsep = 0 em, label=$(\arabic*)$]
\item $\cF(\cE)$ is closed under extensions, submodules and $\tau^{-1}$.
\item $\cF(\cE)$ contains all preprojective modules.
\item For each regular component $\cC$ the set $\cF(\cE) \cap \cC$ forms a non-empty cone in $\cC$, which consists of the successors of a uniquely determined quasi-simple module in $\cC$.
\end{enumerate}
\end{proposition}

\noindent Note that $\cF(\cE)$ is a torsion free class class of some torsion pair $(\cT,\cF(\cE))$ and $\cT(\cE)$ is the torsion class of some torsion pair $(\cT(\cE),\cF)$.

\begin{Lemma}\label{elementary}
Let $M,N$ be indecomposable modules with $\dimu M = (c,1), \dimu N = (1,c)$, $1 \leq c < r$. Then the following statements hold.
\begin{enumerate}[topsep=0em, itemsep= 0em, parsep = 0 em, label=$(\alph*)$]
\item $\tau^z M$ and $\tau^z N$ are elementary for all $z \in \ZZ$. Moreover every proper factor of $M$ is injective and every proper submodule of $N$ is projective.
\item Every proper factor module of $\tau^l M$ is preinjective and every proper submodule of $\tau^{-l} N$ is preprojective for $l \in \NN_0$. 

\end{enumerate}
\end{Lemma}
\begin{proof}
We will give the proofs for $M$. The statements for $N$ follow by duality.\\
$(a)$ By \cite[1.1]{KerLuk1} it is enough to show that $M$ is elementary. Since $(c,1)$ is an imaginary root of the quadratic form, $M$ is regular. Now let $0 \subset X \subset M$ be a proper submodule with dimension vector $\dimu X = (a,b)$. Then $b = 1$ since $\Hom(\cI,\cR)= 0$. Hence $\dimu M/X = (c-a,0)$ and $M/X$ is injective.\\
$(b)$ Let $l \geq 1$. Since a proper factor $\tau^l M \to X$ with $X$ regular induces a proper regular factor $M \to \tau^{-l} X$, the statement follows by $(a)$.
\end{proof}

\subsection{An algebraic family of test-modules}
Let $r \geq 2$. Now we take a closer look at the modules $(X_\alpha)_{\alpha \in k^r\setminus \{0\}}$. Let us start this section by recalling some definitions from \cite{Wor1} and the construction of the module family. We use a slighty different notation, since we are only interested in the case $B(2,r) = \cK_r$.
\noindent For $\alpha \in k^{r}$ and $M \in \modd \cK_r$ we define $x_{\alpha} := \alpha_1 \gamma_1 + \ldots + \alpha_r \gamma_r$ and denote with $x^M_{\alpha} \colon M \to M$ the linear operator associated to $x_\alpha$. 

\begin{Definition}
For $\alpha \in k^r \setminus \{0\}$, the map $\overline{\alpha} \colon \langle e_2 \rangle_k = P_1 \to P_2, e_2 \mapsto \alpha_1 \gamma_1 + \ldots + \alpha_r \gamma_r = x_\alpha$ defines an embedding of $\cK_r$-modules. We now set $X_\alpha := \coker \overline{\alpha}$.
\end{Definition}
These modules are the "test"-modules introduced in \cite{Wor1}. In fact $\im \overline{\alpha}$ is a $1$-dimensional submodule of $P_2$ contained inside the radical $\rad(P_2)$ of the local module $P_2$.   From the definition we get an exact sequence $0 \to P_1 \to P_2 \to X_\alpha \to 0$ and $\dim X_\alpha = (1,r)-(0,1) = (1,r-1)$. Since $P_2$ is local with semi-simple radical $\rad(P_2)=P^r_1$ it now seems natural to study embeddings $P^d_1 \to P_2$ for $1\leq d < r$ and the corresponding cokernels. This motivates the next definition. We restrict ourselves to $d < r$ since otherwise the cokernel is the simple injective module. 

\begin{Definition}
Let and $1 \leq d < r$. For $T = (\alpha_1,\ldots,\alpha_d) \in (k^r)^d$ we define $\overline{T} \colon (P_1)^d \to P_2$ as the $\cK_r$-linear map 
\[ \overline{T}(x) = \sum^d_{i=1} \overline{\alpha}_i \circ \pi_{i}(x),\]
where $\pi_i \colon (P_1)^d \to P_1$ denotes the projection onto the $i$-th coordinate. 
\end{Definition}

The map $\overline{T}$ is injective if and only if $T$ is linearly independent; then we have $\dimu \coker \overline{T} = \dimu P_2 - d \dimu P_1 = (1,r-d)$ and $\coker \overline{T}$ is indecomposable, because $P_2$ is local. Moreover $(1,r-d)$ is an imaginary root of $q$ and therefore $\coker \overline{T}$ is regular indecomposable and by $\ref{elementary}$ elementary. We define $\langle T \rangle := \langle \alpha_1,\ldots,\alpha_d \rangle_k$.

\begin{Lemma}
Let $T,S \in (k^r)^d$ be linearly independent, then $\coker \overline{T} \cong \coker \overline{S}$ if and only if $\langle T \rangle = \langle S \rangle$.
\end{Lemma}
\begin{proof}
If  $\langle T \rangle = \langle S \rangle$, then the definition of $\overline{T}$ and $\overline{S}$ implies that $\im \overline{T} = \im \overline{S}$. Hence $\coker \overline{T} = P_2/\im{\overline{T}} = \coker \overline{S}$. Now let $\langle S \rangle \neq \langle T \rangle$ and assume that $0 \neq \varphi \colon \coker \overline{T} \to \coker \overline{S}$ is $\cK_r$-linear. Since $\coker \overline{S}$ is local with radical $P^{r-d}_1$ and $\Hom(\cR,\cP) = 0$ the map $\varphi$ is surjective and therefore injective. Recall that $P_2$ has $\{e_1,\gamma_1,\ldots,\gamma_r\}$ as a basis. Let $x \in P_2$ such that 
$\varphi(e_1 + \im \overline{T})  = x + \im \overline{S}$. Since $\varphi$ is $\cK_r$-linear we get 
\[  x + \im \overline{S} = e_1  \varphi(e_1 + \im \overline{T}) = e_1 x + \im \overline{S},\]
and hence $x-e_1 x \in \im \overline{S}$. Write $x = \mu e_1 + \sum^r_{i=1} \mu_i \gamma_i$, then $x - \mu e_1 =  \sum^r_{i=1} \mu_i \gamma_i = x - e_1 x \in \im  \overline{S}$ and $x + \im \overline{S}= \mu e_1 + \im \overline{S}$.\\
The assumption $\langle S \rangle \neq \langle T \rangle$ yields  $y \in \im \overline{S}\setminus \im \overline{T} \subseteq \langle \gamma_1,\ldots,\gamma_r\rangle_k$. Then $y + \im \overline{T} \neq 0$ and 
\[ \varphi(y + \im \overline{T}) = y \varphi(e_1 + \im \overline{T}) = \mu y (e_1 + \im \overline{S}) = \mu y + \im \overline{S} = \im \overline{S},\]
a contradiction to the injectivity of $\varphi$. Hence $\Hom(\coker \overline{T},\coker \overline{S}) = 0$ and $\coker \overline{T} \not\cong \coker \overline{S}$.
\end{proof}

\begin{Definition}
Let $r \geq 2$ and $U \in \Gr_{d,r}$ with basis $T = (u_1,\ldots,u_d)$, we define $X_U := \coker \overline{T}$.
\end{Definition}

\begin{Remark} $X_U$ is well defined $($up to isomorphism$)$ with dimension vector $\dim X_U =  (1,r-d)$. $X_U$ is elementary for $r \geq 3$ and quasi-simple for $r=2$.
\end{Remark}

For a module $X$ we define $\add X$ as the category of summands of finite direct sums of $X$ and $Q^d$  denotes the set of isomorphism classes $[M]$ of indecomposable modules $M$ with dimension vector $(1,r-d)$ for $1 \leq d < r$.

\begin{proposition}\label{dim}
Let $M$ be indecomposable.
\begin{enumerate}[topsep=0em, itemsep= -0em, parsep = 0 em, label=$(\alph*)$]	
\item If $[M] \in Q^d$, then there exists $U \in \Gr_{d,r}$ with $M \cong X_U$.
\item The map $\varphi \colon \Gr_{d,r} \to Q^d;  U \mapsto [X_U]$,
is bijective. 
\item Let $1 \leq c \leq d < r$ and $[M] \in Q^d$. There is $[N] \in Q^c$ and an epimorphism $\pi \colon N \to M$.
\end{enumerate}
\end{proposition}
\begin{proof}
$(a)$ Let $0 \subsetneq X \subsetneq M$ be a submodule of $M$. Then $X \subseteq \rad(M) = {P^{r-d}_1}$ and $X$ is in $\add P_1$.
It is $1 = \dim_k M_1 = \dim_k \Hom(\cK_r e_1,M) = \dim_k \Hom(P_2,M)$, so we find a non-zero map $\pi \colon P_2 \to M$. 
Since every proper submodule of $M$ is in $\add P_1$ and $\Hom(P_2,P_1) = 0$, the map $\pi \colon P_2 \to M$ is surjective and yields an exact sequence $0 \to P^d_1 \stackrel{\iota}\to P_2 \stackrel{\pi}{\to} M \to 0$.
For $1 \leq i \leq d$ there exist uniquely determined elements $\beta_i,\alpha^1_i,\ldots,\alpha^r_i \in k$ such that $\iota(g_i(e_2)) = \beta_i e_1 + \alpha^1_i \gamma_1 + \ldots + \alpha^r_i \gamma_r \in P_2 = \langle \gamma_i,e_1 \mid 1 \leq i \leq r \rangle_k$, where $g_i \colon P_1 \to P^d_1$ denotes the embedding into the $i$-th coordinate. Since $e_2$ is an idempotent with $e_2 \gamma_j = \gamma_j$ $(1 \leq j \leq r)$  and $e_2 e_1 = 0$   we get
\[ \alpha^1_i \gamma_1 + \ldots + \alpha^r_i \gamma_r = e_2 (\iota \circ g_i(e_2)) = (\iota \circ g_i)(e_2\cdot e_2) = (\iota \circ g_i)(e_2) = \beta_i e_1 + \alpha^1_i \gamma_1 + \ldots + \alpha^r_i \gamma_r.\]
Hence $\beta_i = 0$. Now define $\alpha_i := (\alpha^1_i,\ldots,\alpha^r_i)$, $T := (\alpha_1,\ldots,\alpha_d)$ and $U := \langle T \rangle$. It is $\iota = \overline{T}$ and by the injectivity of $\iota$ we conclude that $T$ is linearly independet and therefore $U \in \Gr_{d,r}$. Now we conclude $X_U = \coker \overline{T} = \coker \iota = M$. \\
$(b)$ This is an immediate consequence of $(a)$.\\
$(c)$ By $(a)$ we find $U$ in $\Gr_{d,r}$ with basis $T = (u_1,\ldots,u_d)$ such that $X_U \cong M$.  Let $V$ be the subspace with basis $S = (u_1,\ldots,u_c)$. Then $\im \overline{S} \subseteq \im \overline{T}$ and we get an epimorphism $\pi \colon X_V = P_2 / {\im \overline{S}} \to P_2 /{\im \overline{T}} = X_U, x + \overline{S} \to x + \overline{T}$ with $\dimu X_V = (1,r-c)$. 
\end{proof}

As a generalization of $x^M_\alpha \colon M \to M$ we introduce maps $x^M_T \colon M \to M^d$ and $y^M_T \colon M^d \to  M$ for $1 \leq d < r$ and $T \in (k^r)^d$. Note that $x^M_T = y^M_T$ if and only if $d = 1$.

\begin{Definition}
Let $1 \leq d < r$ and $T = (\alpha_1,\ldots,\alpha_d) \in (k^r)^d$. We denote by $x^M_T$ and $y^M_T$ the operators \[ x^M_T \colon M \to M^d, m \mapsto (x^M_{\alpha_1}(m),\ldots,x^M_{\alpha_d}(m)),\]
\[ y^M_T \colon M^d \to M, (m_1,\ldots,m_d) \mapsto x^M_{\alpha_1}(m_1)+\ldots+x^M_{\alpha_d}(m_d).\]
\end{Definition}

\noindent It is $\im x^M_T \subseteq M_2 \oplus \ldots\oplus M_2$, $M_2 \oplus \ldots \oplus M_2 \subseteq \ker y^M_T $ and $D(x^M_T) = y^{DM}_T$ since for $f = (f_1,\ldots,f_d) \in (DM)^d$ and $m \in M$ we have
\begin{align*}
D(x^M_T)(f)(m)& = D(x^M_T)(f_1,\ldots,f_d)(m) = \sum^d_{i=1} (f_i \circ x^M_{\alpha_i})(m)\\ 
&= \sum^d_{i=1} f_i(x_{\alpha_i}.m) 
= \sum^d_{i=1} (x_{\alpha_i}.f_i)(m) \\
&= y^{DM}_T(f_1,\ldots,f_d)(m) = y^{DM}_T(f)(m).
\end{align*}

\begin{Lemma}\label{LemmaQuotient}
Let $1 \leq d < r$ and $U \in \Gr_{d,r}$. Every non-zero quotient $Q$ of $X_U$ is indecomposable. $Q$ is preinjective $($injective$)$ if $\dimu Q = (1,0)$ and regular otherwise.
\end{Lemma}
\begin{proof}
Since $X_U$ is regular, we conclude with $\Hom(\cR,\cP) = 0$, that every indecomposable non-zero quotient of $X_U$ is preinjective or regular. Let $Q$ be such a quotient with $\dimu Q = (a,b)$ and $Q \neq X_U$. Since $X_U$ is local with radical ${P^{r-d}_1}$ and $\dimu = (1,r-d)$, it follows $(1,r-d) = (a,b) + (0,c)$ for some $c > 0$. Hence $a = 1$ and $Q$ is an injective module if $b = 0$. Otherwise $Q$ is also indecomposable since $b > 0$ and $Q = A \oplus B$ with $A,B \neq 0$ imply w.l.o.g $(\dimu B)_1 = 0$. Hence $B \in \add P_1$ which is a contradiction to $\Hom(\cR,\cP) = 0$.
\end{proof}

\subsection{Modules for the generalized Kronecker algebra}
In the following we will give the definition of $\cK_r$-modules $(r \geq 2)$ with constant radical rank and constant socle rank.
\begin{Definition}
Let $M$ be in $\modd \cK_r$ and $1 \leq d < r$.
\begin{enumerate}[topsep=0em, itemsep= -0em, parsep = 0 em, label=$(\alph*)$]	
\item $M$ has \textbf{constant d-radical rank} if the dimension of 
\[\Rad_U(M) := \sum_{u \in U} x^M_u(M) \subseteq M_2\]
is independent of the choice of $U \in \Gr_{d,r}$.
\item $M$ has \textbf{constant d-socle rank} if the dimension of 
\[ \Soc_U(M) := \{ m \in M \mid \forall u \in U : \ x^M_u(M) = 0 \} = \bigcap_{u \in U} \ker(x^M_u) \supseteq M_2 \]
is independent of the choice of $U \in \Gr_{d,r}$.
\item $M$ has the \textbf{equal d-radical property} if $\Rad_U(M)$ is independent of the choice of $U \in \Gr_{d,r}$
\item $M$ has the \textbf{equal d-socle property} if $\Soc_U(M)$ is independent of the choice of $U \in \Gr_{d,r}$
\end{enumerate}
The corresponding full subcategories of $\modd \cK_r$ are denoted by $\CRR_d$, $\CSR_d$, $\ERP_d$ and $\ESP_d$.
\end{Definition}

Note that $\CRR_1 = \CR = \CSR_1$, $\ERP_1 = \EIP$, $\ESP_1 = \EKP$ and for $U \in \Gr_{d,r}$ with basis $(u_1,\ldots,u_d)$ we have $\Rad_U(M) = \sum^d_{i=1} x^M_{u_i}(M)$ and $\Soc_{U}(M) = \bigcap^d_{i=1} \ker(x^M_{u_i})$.  We restrict the definition to $d < r$ since $\Gr_{r,r} = \{k^r\}$ and therefore every module in $\modd \cK_r$ is of constant $r$-socle and  $r$-radical rank. 

\begin{Lemma}\label{Radical}
Let $M$ be indecomposable and not simple. Then
\begin{enumerate}[topsep=0em, itemsep= -0em, parsep = 0 em, label=$(\alph*)$]	

\item $M \in \ESP_d$ if and only if $M_2 = \Soc_U(M)$ for all $U \in \Gr_{d,r}$.
\item $M \in \ERP_d$ if and only if $M_2 = \Rad_U(M)$ for all $U \in \Gr_{d,r}$.
\end{enumerate} 

\end{Lemma}
\begin{proof}
$(a)$ Assume that $M$ is in $\ESP_d$ and let $W:= \Soc_{U}(M)$ for $U \in \Gr_{d,r}$.
Denote by $e_1,\ldots,e_r \in k^r$ the canonical basis vectors. Since $M$ is not simple it is $($\cite[5.1.1]{Far2}$)$  $\bigcap^r_{i=1} \ker (x^M_{e_i}) = M_2$. Denote by $S(d)$ the set of all subsets of $\{1,\ldots,r\}$ of cardinality $d$. Then
\[ \bigcap_{S \in S(d)} \bigcap_{j \in S}  \ker(x^M_{e_j}) =  \bigcap^{r}_{i=1} \ker(x^M_{e_i}) = M_2.\]
Since $\langle e_j \mid j \in S\rangle_k \in \Gr_{d,r}$ and $M \in \ESP_d$ we get $\bigcap_{j \in S}  \ker(x^M_{e_j}) = W$ and hence $M_2 = \bigcap_{S \in S(d)} W = W$.\\
$(b)$ Let $M \in \ERP_d$, $U \in \Gr_{d,r}$ and $W:= \Rad_U(M)$. Since $M$ is not simple it is $($\cite[5.1.1]{Far2}$)$ $\sum^{r}_{i=1} x^M_{e_i} = M_2$ and hence 
\[  W = \sum_{S \in S(d)} \sum_{j \in S}  x^M_{e_j}(M) =  \sum^{r}_{i=1} x^M_{e_i}(M) = M_2.\]
\end{proof}

\begin{Lemma}\label{Dual}
Let $M$ be in $\modd \cK_r$ and $1 \leq d < r$ .
\begin{enumerate}[topsep=0em, itemsep= -0em, parsep = 0 em, label=$(\alph*)$]	
\item $M \in \CSR_d$ if and only if $DM \in \CRR_d$.
\item $M \in \ESP_d$ if and only if $DM \in \ERP_d$.
\end{enumerate}
\end{Lemma}
\begin{proof}
Note that $\Rad_U(DM) = \im(y^{DM}_T) = \im D(x^M_T) = D(\im x^M_T)$ and hence 
\[\dim_k M - \dim_k \Soc_U(M) = \dim_k M - \dim_k \ker x^M_T = \dim_k \im x^M_T = \dim_k D(\im x^M_T) = \Rad_U(DM).\]
 Hence $M \in \CSR_d$ if and only if $DM \in \CRR_d$. Moreover $M$ in $\ESP_d$ if and only if $\Soc_U(M) = M_2$ and hence $\dim_k \Rad_U(DM) = \dim_k M_1 = \dim_k (DM)_2$. 
\end{proof}

\noindent For the proof of the following proposition we use the same methods as in \cite[2.5]{Wor1}.
\begin{proposition}\label{HomChar}
Let $1 \leq d < r \in \NN$. Then
\begin{align*}
\ESP_d &= \{ M \in \modd \cK_r \mid \forall U \in \Gr_{d,r}: \Hom(X_U,M) = 0\}, \\
\CSR_d  &= \{ M \in \modd \cK_r \mid \exists c \in \NN_0 \ \forall U \in \Gr_{d,r}: \dim_k \Hom(X_U,M) = c \}.
\end{align*}
\end{proposition}
\begin{proof}
Let $U \in \Gr_{d,r}$ with basis $T =(\alpha_1,\ldots,\alpha_d)$.
Consider the short exact sequence $0 \to(P_1)^d \stackrel{\overline{T}}{\to} P_2 \to X_U \to 0$. Application of $\Hom(-,M)$ yields 
\[ 0 \to \Hom(X_U,M) \to \Hom(P_2,M) \stackrel{{\overline{T}}^{\ast}}{\to} \Hom(P^d_1,M) \to \Ext(X_U,M) \to 0.\]
Moreover let 
\[f \colon \Hom(P_2,M) \to M_1; g \to g(e_1)\]
 and \[g \colon \Hom(P^d_1,M) \to M^d_2; h \mapsto (h\circ \iota_1(e_2),\ldots,h\circ \iota_d(e_2))\]
be the natural isomorphisms, where $\iota_i \colon P_1 \to P^d_1$ denotes the embedding into the $i$-th coordinate. Let $\pi_{M^d_2} \colon M^d \to M^d_2$ the natural projection, then $g \circ \overline{T}^\ast = \pi_{M^d_2} \circ{x^M_T}_{|M_1} \circ f$. 
 Hence $\dim_k \ker(\pi_{M^d_2} \circ {x^M_T}_{|M_1} \colon M_1 \to M^d_2) =\dim_k \ker(\overline{T}^\ast) = \dim_k \Hom(X_U,M)$. Now let $c \in \NN_0$. We conclude 
\begin{align*}
  \dim_k \Hom(X_U,M) = c &\Leftrightarrow \dim_k \ker(\pi_{M^d_2} \circ {x^M_T}_{|M_1}) = c 
  \stackrel{\im x^M_T \subseteq M^d_2}{\Leftrightarrow} \dim_k \ker ({x^M_T}_{|M_1}) = c \\
  &\stackrel{M_2 \subseteq \ker(x^M_T)}{\Leftrightarrow} \dim_k \ker ({x^M_T}) = c + \dim_k M_2
    \\
   &\stackrel{\Soc_U(M) = \ker(x^M_T)}{\Leftrightarrow} \dim_k \Soc_{U}(M) =  c + \dim_k M_2.
 \end{align*}  

This finishes the proof for $\CSR_d$. Moreover note that $c = 0$ together with Lemma $\ref{Radical}$ yields
 \begin{align*}
 M \in \ESP_d &\Leftrightarrow \exists W \leq M \ \forall  U \in \Gr_{d,r}: \Soc_{U}(M) = W \\
 &\Leftrightarrow \forall  U \in \Gr_{d,r}: \Soc_{U}(M) = M_2  
 \Leftrightarrow \forall  U \in \Gr_{d,r}: \dim_k \Soc_{U}(M) = 0 + \dim_k M_2 \ \ \\
 &\Leftrightarrow \ \forall  U \in \Gr_{d,r}: \Hom(X_U,M) = 0.
\end{align*}  
\end{proof}
  
\noindent Since $\tau \circ D = D \circ \tau^{-1}$ the next result follows from the Auslander-Reiten formula and $\ref{Dual}$.
 
 \begin{proposition}
 Let $1 \leq d < r \in \NN$. Then
\begin{align*}
\ERP_d  &= \{ M \in \modd \cK_r \mid \forall U \in \Gr_{d,r}: \Ext(D \tau X_U,M) = 0  \} \ \text{and}\\
\CRR_d &= \{ M \in \modd \cK_r \mid \exists c \in \NN_0 \ \forall U \in \Gr_{d,r}: \dim_k \Ext(D \tau X_U,M) = c  \}.
\end{align*}
  \end{proposition}

\begin{Remark} Note that for $d = 1$ we have $U = \langle \alpha \rangle_k$ with $\alpha \in k^r \setminus \{0\}$, $X_U \cong X_\alpha$ and $D  \tau X_U \cong X_U$ \cite[3.1]{Wor1}. However, this identity holds if and only if $d = 1$. It follows immediatly from $\ref{HomChar}$ and $\ref{elementary}$ that for $1 \leq d < r-1$ and $V \in \Gr_{d,r}$ the module $X_V$ is in $\CSR_{d+1} \setminus \CSR_{d}$. If not stated otherwise, we assume from now on that $r \geq 3$.
\end{Remark}

\noindent In view of $\ref{Cone2}$, $\ref{elementary}$ and the definitions of $\ESP_d$ and $\ERP_d$ we immediately get:

\begin{proposition}\label{HomTau} Let $1 \leq d < r$ and $\cC$ a regular component of $\cK_r$.
\begin{enumerate}[topsep=0em, itemsep= -0em, parsep = 0 em, label=$(\alph*)$]	
\item $\ESP_1 \subseteq \ESP_2 \subseteq \ldots \subseteq \ESP_{r-1}$ and $\ERP_1 \subseteq \ERP_2 \subseteq \ldots \subseteq \ERP_{r-1}$.
\item $\ESP_d$ is closed under extensions, submodules and $\tau$. Moreover $\ESP_d$ contains all preprojective modules and $\ESP_d \cap \cC$ forms a non-empty cone in $\cC$.
\item $\ERP_d$ is closed under extensions, images and $\tau$. Moreover $\ERP_d$ contains all preinjective modules and $\ERP_d \cap \cC$ forms a non-empty cone in $\cC$.
\end{enumerate}
\end{proposition}

\begin{Definition}
For $1 \leq i < r$ we set $\Delta_i := \ESP_i \setminus \ESP_{i-1}$ and $\nabla_i := \ERP_i \setminus \ESP_{i-1}$, where $\ESP_0 = \emptyset = \ERP_0$.
\end{Definition}
\noindent The next result suggests that for each regular component $\cC$ and $1 < i < r$ only a small part of vertices in $\cC$ corresponds to modules in $\Delta_i$. Nonetheless we will see in Section $5$ that for $1 \leq i < r$ the categories $\Delta_i$ and $\nabla_i$ are of wild type. 
\begin{proposition}\label{Summary}
Let $r \geq 3$, $\cC$ be a regular component and $M_i,W_i$ $(1 \leq i < r)$ in $\cC$ the uniquely determined quasi-simple modules such that $\ESP_i \cap \cC = (M_i \to)$ and $\ERP_i \cap \cC = (\to W_i)$. 
\begin{enumerate}[topsep=0em, itemsep= -0em, parsep = 0 em, label=$(\alph*)$]
\item There exists at most one number  $1 < m(\cC) < r$ such that $\Delta_{m(\cC)} \cap \cC$ is non-empty. If such a number exists then $\Delta_{m(\cC)} \cap \cC = \{ M_{m(\cC)}[l] \mid l \geq 1 \}$.
\item There exists at most one number  $1 < w(\cC) < r$ such that $\nabla_{w(\cC)} \cap \cC$ is non-empty. If such a number exists then $\nabla_{w(\cC)} \cap \cC = \{ (l)W_{w(\cC)} \mid l \geq 1 \}$.
\end{enumerate} 
\end{proposition}
\begin{proof}
$(a)$ Let $M$ be in $\cC$ and assume $M \notin \ESP_{1}$. Then there exists $\alpha \in k^{r}\setminus \{0\}$ with $\Hom(X_U,M) \neq 0$ for $U = \langle\alpha \rangle_k$. Hence we find a non-zero map $f \colon \tau X_U \to \tau M$. Consider an exact sequence $0 \to {P^{r-2}_1} \to  X_U \to N \to 0$. Then $\dimu N = (1,1)$ and by $\ref{LemmaQuotient}$ $N$ is indecomposable. By $\ref{dim}$, there exists $V \in \Gr_{r-1,r}$ with $X_V \cong DN$. Since $DX_U = \tau X_U$ \cite[3.1]{Wor1} we get a non-zero morphism $g \colon X_V \to \tau X_U$ and by $\ref{CompElementary}$ a non-zero morphism 
\[ X_V \stackrel{g}{\to} \tau X_U \stackrel{f}{\to} \tau M.\]
Therefore $\tau M \notin \ESP_{r-1}$ by $\ref{HomChar}$. Now assume that $M_i \neq M_{j}$ for some $i$ and $j$. Then in particular $M_1 \neq M_{r-1}$. Hence $M_{r-1} = \tau^l M_1$ for some $l \geq 1$. By definition we have $M:= \tau M_{1} \notin \ESP_1$ and the above considerations yield $\tau(\tau M_1) = \tau M \notin \ESP_{r-1}$. Therefore $1 \leq l < 2$ since $\ESP_{r-1} \cap \cC$ is closed under $\tau^{-1}$. Therefore $M_{r-1} = \tau M_1$.\\
$(b)$ This follows by duality.
\end{proof}

We state two more results that follow from $\ref{HomChar}$ and will be needed later on. The first one is a generalization of \cite[3.5]{Wor1} and follows with the same arguments.

\begin{Lemma}\label{Allquasi}
Let $0 \to A \to B \to C \to 0$ be an almost split sequence such that two modules of the sequence are of constant $d$-socle rank. Then the third module also has constant $d$-socle rank.
\end{Lemma}

\begin{Definition}
Let $r \geq 2$, $1 \leq d < r$ and $\mathfrak{X}_{d,r} := \{ X_U \mid U  \in \Gr_{d,r} \}$. Let $\mathfrak{X}^\perp_{d,r}$ be the right orthogonal category $\mathfrak{X}^\perp_{d,r} = \{M \in \modd \cK_r \mid \forall U \in \Gr_{d,r} \colon \Hom(X_U,M) = 0\}$
and  $^\perp\mathfrak{X}_{d,r}$ be the left orthogonal category
$^\perp\mathfrak{X}_{d,r} := \{M \in \modd \cK_r \mid \forall U \in \Gr_{d,r} \colon \Hom(M,X_U) = 0\}$. Then we set $\overline{\mathfrak{X}}_{d,r} := \mathfrak{X}^\perp_{d,r} \cap {^\perp\mathfrak{X}_{d,r}}$.
\end{Definition}

\noindent Note that every module in $\overline{\mathfrak{X}}_{d,r}$ is regular by $\ref{HomTau}$.

\begin{Lemma}\label{Finalall}
Let $r \geq 3$, $1 \leq d < r$ and $M$ be quasi-simple regular in a component $\cC$ such that $M \in \overline{\mathfrak{X}}_{d,r}$. Then every module in $\cC$ has constant $d$-socle rank.
\end{Lemma}
\begin{proof}
Let $V \in \Gr_{d,r}$. We have shown in $\ref{HomTau}$ that the set $\{ N \mid \Hom(X_V,N) = 0\}\cap \cC$ $($resp. $\{ N \mid \Ext(X_V,N) = 0\} \cap \cC)$ is closed under $\tau^{-1}$ $($resp. $\tau)$. Since $0 = \dim_k \Hom(M,X_V) = \dim_k \Ext(X_V,\tau M)$ we have $\Ext(X_V,\tau^l M) = 0$  for $l \geq 1$.  The Euler-Ringel form yields
\[ 0 =  \dim_k \Ext(X_V,\tau^l M) = -\langle \dimu X_V,\dimu \tau^l M\rangle + \dimu \Hom(X_V,\tau^l M).\] 
Since $\langle \dimu X_V,\dimu \tau^l M\rangle = \langle (1,r-d),\dimu M \rangle$ is independent of $V$, $\tau^l M$ has constant $d$-socle rank. On the other hand $\Hom(X_V,M) = 0$ implies that $\tau^{-q} M$ has constant $d$-socle rank for all $q \geq 0$. It follows that each quasi-simple module in $\cC$ has constant $d$-socle rank. Now apply $\ref{Allquasi}$.
\end{proof}

\section{Process of Simplification and Applications}

\subsection{Representation type}
Denote by $\Lambda := kQ$ the path algebra of a connected, wild quiver $Q$. We stick to the notation introduced in \cite{Ker2}. Recall that a module $M$ is called brick if $\End(M) = k$ and $M,N$ are called orthogonal if $\Hom(M,N) = 0 = \Hom(N,M)$.

\begin{Definition}
Let $\cX$ be a class of pairwise orthogonal bricks in $\modd \Lambda$. The full subcategory $\cE(\cX)$ is by definition the class of all modules $Y$ in $\modd \Lambda$ with an $\cX$-filtration, that is, a chain 
\[ 0 = Y_0 \subset Y_1 \subset \ldots \subset Y_{n-1} \subset Y_n = Y \]
with $Y_i/Y_{i-1} \in \cX$ for all $1 \leq i \leq n$.
\end{Definition}
 In \cite[1.]{Ri4} the author shows that $\cE(\cX)$ is an exact abelian subcategory of $\modd \Lambda$, closed under extensions, and $\cX$ is the class of all simple modules in $\cE(\cX)$. In particular, a module $M$ in $\cE(\cX)$ is indecomposable if and only if it is indecomposable in $\modd \cK_r$.

\begin{proposition}
\label{Simplification}
Let $r \geq 3$ and $\cX \subseteq \modd \cK_r$ be a class of pairwise orthogonal bricks with self-extensions $($and therefore regular$)$.
\begin{enumerate}[topsep=0em, itemsep= -0em, parsep = 0 em, label=$(\alph*)$]	
\item Every module in $\cE(\cX)$ is regular.
\item Every regular component $\cC$ contains at most one module of $\cE(\cX)$.
\item Every indecomposable module $N \in \cE(\cX)$ is quasi-simple in $\modd \cK_r$.
\item $\cE(\cX)$ is a wild subcategory of $\modd \cK_r$
\end{enumerate}
\end{proposition}
\begin{proof}
$(a)$ and $(b)$ are proven in \cite[1.1, 1.4]{Ker2} for any wild hereditary algebra and $(c)$ follows by \cite[1.4]{Ker2} and the fact that every regular brick in $\modd \cK_r$ is quasi-simple \cite[9.2]{Ker3}. Let $M \in \cX$, then $t := \dim_k \Ext(M,M) \geq 2$ by $\ref{Extgeq2}$. Due to \cite[3.7]{Gab1} and \cite[Remark 1.4]{Ker2} the category $\cE(\{M\}) \subseteq \cE(\cX)$ is equivalent to the category of finite dimensional modules over the power-series ring $k\langle \langle X_1,\ldots,X_t \rangle \rangle$ in non-commuting variables $X_1,\ldots,X_t$. Since $t \geq 2$ the category $\cE(\{M\}) \subseteq \cE(\cX)$ is wild, and also $\cE(\cX)$.
\end{proof}

We will use the above result to prove the existence of numerous components, such that all of its vertices correspond to modules of constant $d$-socle rank. By duality, all results also follow for constant radical rank. As a by-product we verify the wildness of $\EKP = \ESP_1$ and $\EIP = \ERP_1$. Using the functor $\fF \colon \modd \cK_r \to \modd kE_r$, we show the wildness of the corresponding full subcategories in $\modd_2 kE_r$ of $E_r$-modules of Loewy length $\leq 2$.\\

\subsection{Passage between $\cK_r$ and $\cK_s$}
Let  $2 \leq r < s \in \NN$. Denote by $\inf^s_r \colon \modd \cK_r \to \modd \cK_s$ the functor that assigns to a $\cK_r$-module $M$ the module $\inf^s_r(M)$ with the same underlying vector space so that the action of $e_1,e_2,\gamma_1,\ldots,\gamma_r$ on $\inf^s_r(M)$ stays unchanged and all other arrows act trivially on $\inf^s_r(M)$.
Moreover let $\iota \colon \cK_r \to \cK_s$ be the natural $k$-algebra monomorphism given by $\iota(e_i) = e_i$ for $i \in \{1,2\}$ and $\iota(\gamma_j) = \gamma_j$. Then each $\cK_s$-module $N$ becomes a $\cK_r$-module $N^\ast$ via pullback along $\iota$. Denote the corresponding functor by $\res^s_r \colon \modd \cK_s \to \modd \cK_r$. In the following $r,s$ will be fixed, so we suppress the index and write just $\inf$ and $\res$.

\begin{Lemma}\label{functorG}
Let $2 \leq r < s \in \NN$. The functor $\inf \colon \modd \cK_r \to \modd \cK_s$ is fully faithful and exact. The essential image of $\inf$ is a subcategory of $\modd \cK_s$ closed under factors and submodules. Moreover $\inf(M)$ is indecomposable if and only if $M$ is indecomposable in $\modd \cK_r$.
\end{Lemma}
\begin{proof}
Clearly $\inf$ is fully faithful and exact. Now let $M \in \modd \cK_r$ and $U \subseteq \inf(M)$ be a submodule. Then $\gamma_i$ $(i > r)$ acts trivially on $U$ and hence the pullback $\res(U) =: U^\ast$ is a $\cK_r$-module with $\inf(U^\ast) = \inf \circ \res(U) = U$.
Now let $f \in \Hom_{\cK_s}(\inf(M),V)$ be an epimorphism. Let $v \in V$ and $m \in M$ such that $f(m) = v$. It follows
$\gamma_i v = \gamma_i f(m) = f(\gamma_i m) = 0$ for all $i > r$. This shows that $\inf(V^\ast) = \inf \circ \res(V) = V$.\\
Since $\inf$ is fully faithful we have $\End_{\cK_s}(\inf(M)) \cong \End_{\cK_r}(M)$. Hence $\End_{\cK_s}(\inf(M))$ is local if and only if $\End_{\cK_r}(M)$ is local. 
\end{proof}

\noindent Statement $(a)$ of the succeeding Lemma is stated in \cite[3.1]{Bo1} without proof.

\begin{Lemma}\label{RegularAndQuasi}
Let $2 \leq r < s$ and let $M$ be an indecomposable $\cK_r$-module that is  not simple. The following statements hold. 
\begin{enumerate}[topsep=0em, itemsep= -0em, parsep = 0 em, label=$(\alph*)$]
\item $\inf(M)$ is regular and quasi-simple.
\item $\inf(M) \not\in \CSR_m$ for all $m \in \{1,\ldots,s-r\}$.
\end{enumerate}
\end{Lemma}
\begin{proof}
$(a)$ Write $\dimu M = (a,b) \in \NN_0 \times \NN_0$. Since $M$ is not simple we have $ab \neq 0$ and $q(\dimu M) = a^2 + b^2 - r ab \leq 1$. It follows $q(\dimu \inf(M)) = a^2 + b^2 - sab = a^2 + b^2 -rab -(s-r) ab  \leq 1 -(s-r)ab < 1$. Hence $q(\dimu \inf(M)) \leq 0$ and $\inf(M)$ is regular.\\
Assume that $\inf(M)$ is not quasi-simple, then $\inf(M) = U[i]$ for $U$ quasi-simple with $i \geq 2$. By $\ref{functorG}$ we have $U[i-1] = \inf(A)$ and $\tau^{-1} U[i-1] = \inf(B)$ for some $A,B$ indecomposable in $\modd \cK_r$. 
Fix an irreducible monomorphism $f \colon \inf(A) \to \inf(M)$. Since $\inf$ is full, we find $g \colon A \to M$ with $\inf(g) = f$. The faithfulness of $\inf$ implies that $g$ is an irreducible monomorphism $g \colon A \to M$. By the same token there exists an irreducible epimorphism $M \to B$. As all irreducible morphisms in $\cP$ are injective and all irreducible morphisms in $\cI$ are surjective, $M$ is located in a $\ZZ A_\infty$ component. It follows that $\tau B = A$ in $\modd \cK_r$. Let $\dimu B = (c,d)$, then the Coxeter matrices for $\cK_r$ and $\cK_s$ yield
\[ ((r^2-1)c -rd,rc-d) = \dimu \tau B = \dimu A = \dimu \inf(A) = \dimu \tau \inf(B) = ((s^2-1)c-sd,sc-d).\]
This is a contradiction since $s \neq r$.\\
$(b)$ Denote by $\{e_1,\ldots,e_s\}$ the canonical basis of $k^s$. Let $1 \leq m \leq s-r$ and set $U := \langle e_{r+1},\ldots,e_{r+m} \rangle_k$. Then $\Soc_{U}(M) = \bigcap^m_{i=1} \ker (x^M_{e_{r+i}}) = M$. Let $j \in \{1,\ldots,r\}$ such that $\gamma_j$ acts non-trivially on $M$. Let $V \in \Gr_{m,s}$ such that $e_j \in V$. Then $\Soc_{V}(M) \neq M$ and $M$ does not have constant $m$-socle rank.
\end{proof}

\begin{proposition}\label{goingup}
Let $2 \leq r < s \in \NN$, $1 \leq d < r$ and $M$ be an indecomposable and non-simple $\cK_r$-module. Then the following statements hold.
\begin{enumerate}[topsep=0em, itemsep= -0em, parsep = 0 em, label=$(\alph*)$]
\item   If $M \in {^\perp\mathfrak{X}_{d,r}}$ then $\inf(M) \in {^\perp\mathfrak{X}_{d+s-r,s}}$.
\item If $M \in \mathfrak{X}^\perp_{d,r}$ then $\inf(M) \in \mathfrak{X}^\perp_{d+s-r,s}$.
\item If $M \in \overline{\mathfrak{X}}_{d,r}$ then $\inf(M)$ is contained in a regular  component $\cC$ with $\cC \subseteq \CSR_{d+s-r}$.
\end{enumerate}
\end{proposition}

\begin{proof}
By definition it is $1 \leq d+s-r < s$. Now fix $V \in \Gr_{d+s-r,s}$ and note that $\dimu X_V = (1,s-(d+s-r)) = (1,r-d)$, which is the dimension vector of every $\cK_r$-module $X_U$ for $U \in \Gr_{d,r}$.\\
$(a)$ Assume that $\Hom(\inf(M),X_V) \neq 0$ and let $0 \neq f \colon \inf(M) \to X_V$. By $\ref{functorG}$ and $\ref{elementary}$ the $\cK_s$-module $\inf(M)$ is regular and every proper submodule of $X_V$ is preprojective. Hence $f$ is surjective onto $X_V$. Again $\ref{functorG}$ yields $Z \in \modd \cK_r$ indecomposable with $\dimu Z = (1,r-d) = \dimu X_V$ such that $X_V = \inf(Z)$. By $\ref{dim}$ there exists $U \in \Gr_{d,r}$ with $Z = X_U$. Since $\inf$ is fully faithful it follows $0 = \Hom(M,X_U) \cong \Hom(\inf(M),\inf(X_U)) = \Hom(\inf(M),X_V) \neq 0$, a contradiction.\\
$(b)$ Assume that $\Hom(X_V,\inf(M)) \neq 0$ and let $f \colon X_V \to \inf(M)$ non-zero. Since $\inf(M)$ is regular indecomposable the module $\im f \subseteq \inf(M)$ is not injective and $\ref{LemmaQuotient}$ yields that $\im f$ is indecomposable and regular. 
As $\im f$ is a submodule of $\inf(M)$ there exists an indecomposable module $Z \in \modd \cK_r$ with $\inf(Z) = \im f$. Since $\im f$ is not simple we have $\dimu \im f = (1,r-c)$ for $1 \leq r-c \leq r-d$. Hence $Z = X_U$ for $U \in \Gr_{c,r}$ and by $\ref{dim}(d)$ there exists $W \in \Gr_{d,r}$ and an epimorphism
$\pi \colon X_W \to X_U$. We conclude with $0 \neq \Hom(\im f,\inf(M)) = \Hom(\inf(X_U),\inf(M)) \cong \Hom(X_U,M)$ and the surjectivity of $\pi \colon X_W \to X_U$ that $\Hom(X_W,M) \neq 0$, a contradiction to the assumption.\\
$(c)$ By $\ref{RegularAndQuasi}$ the module $\inf(M)$ is quasi-simple in a regular component and satisfies the conditions of $\ref{Finalall}$ for $q := d+s-r$ by $(a)$ and $(b)$. 
\end{proof}

\begin{examples}\label{twoexamples}
The following two examples will be helpful later on.
\begin{enumerate}[topsep=0em, itemsep= -0em, parsep = 0 em, label=$(\arabic*)$]
\item Let $r = 3$. Ringel has shown that the representation
$F = (k^2,k^2,F(\gamma_1),F(\gamma_2),F(\gamma_3))$ with the linear maps $F(\gamma_1) = id_{k^2}$, $F(\gamma_2)(a,b) = (b,0)$ and $F(\gamma_3)(a,b) = (0,a)$ is elementary. Let $E$ be the corresponding $\cK_3$-module. Then $\dimu E = (2,2)$ and it is easy to see that every indecomposable submodule of $E$ has dimension vector $(0,1)$ or $(1,2)$. In particular $\Hom(W,E) = 0$ for each indecomposable module with dimension vector $\dimu W = (1,1)$. Assume now that $f \colon E \to W$ is non-zero, then $f$ is surjective since every proper submodule of $W$ is projective. Since $E$ is elementary, $\ker f$ is a preprojective module with dimension vector $(1,1)$, a contradiction. Hence 
 $E \in \overline{\mathfrak{X}}_{2,3}$. 
 \item Given a regular component $\cC$, there are unique quasi-simple modules $M_\cC$ and $W_\cC$ in $\cC$ such that $\ERP_1 \cap \cC = (\to W_\cC)$ and $\ESP_1 \cap \cC = (M_\cC \to)$. The width $\cW(\cC) \in \ZZ$ is defined \cite[3.3, 3.3.1]{Wor1} as the unique integer satisfying $\tau^{\cW(\cC) + 1} M_\cC = W_\cC$. In fact it is shown that $\cW(\cC) \in \NN_0$ and an example of a regular component $\cC$ with $\cW(\cC) = 0$ and $\End(M_\cC) = k$ is given. Since $X_U \cong D \tau X_U$ for $U \in \Gr_{1,r}$ we conclude for an arbitrary regular component $\cC$: 
\begin{align*}
 \cW(\cC) = 0 &\Leftrightarrow \tau M_\cC = W_\cC \Leftrightarrow \tau M_\cC \in \EIP \\
&\Leftrightarrow \Hom(X_U,M_\cC) = 0 = \Ext(X_U,\tau M_\cC) \ \text{for all} \ U \in \Gr_{1,r} \\ 
 & \Leftrightarrow \Hom(X_U,M_\cC) = 0 = \Hom(M_\cC,X_U) \ \text{for all} \ U \in \Gr_{1,r} \Leftrightarrow M_\cC \in \overline{\mathfrak{X}}_{1,r}.
\end{align*}
\end{enumerate}

\end{examples}

\begin{Lemma}\label{NOstable}
Let $s \geq 3$ and $2 \leq d < s$. Then there exists a regular module $E_d$ with the following properties.
\begin{enumerate}[topsep=0em, itemsep= -0em, parsep = 0 em, label=$(\alph*)$]
\item $E_d$ is a $($quasi-simple$)$ brick in $\modd \cK_s$.
\item $E_d \in \overline{\mathfrak{X}}_{d,s}$.
\item There exist $V,W \in \Gr_{1,s}$ with $\Hom(X_V,E_d) = 0 \neq \Hom(X_W,E_d)$.
\end{enumerate}
\end{Lemma}
\begin{proof}
We start by considering $s = 3$ and $d = 2$. Pick the elementary module $E_d := E$ from the preceding example. $E$ is a brick and $E \in \overline{\mathfrak{X}}_{d,s}$. Set $V = \langle \alpha:= (1,0,0) \rangle_k$, $W = \langle \beta:= (0,1,0) \rangle_k$. By the definition of $E$ we have 
\[ \dim_k \ker x^E_{\alpha} = 2 \neq 3 = \dim_k \ker x^E_{\beta},\]
and therefore 
\[ \dim_k \Hom(X_V,E_d) = 0 \neq 1 = \dim_k \Hom(X_W,E_d).\]
\noindent Now let $s > 3$.
If $d = s-1$ consider $E_d := \inf^s_3(E)$. In view of $\ref{goingup}$  we have $E_d \in \overline{\mathfrak{X}}_{2+s-3,s} = \overline{\mathfrak{X}}_{d,s}$
Moreover $\inf(E)$ is a brick in $\modd \cK_s$ and for the canonical basis vectors $e_1,e_2 \in k^s$ and $V = \langle e_1 \rangle_k$, $W := \langle e_2 \rangle_k$ we get as before 
\[ \dim_k \Hom(X_V,\inf(E)) = 0 \neq 1 = \Hom(X_W,\inf(E)).\]
\noindent Now let $1 < d < s-1$. Set $r :=  1 + s - d \geq 3$, consider a regular component for $\cK_r$ with $\cW(\cC) = 0$ such that $M_\cC$ is a brick and set $M := M_\cC$. Then we have
$M \in \overline{\mathfrak{X}}_{1,r}$ and $\ref{goingup}$ yields
$E_d := \inf(M) \in \overline{\mathfrak{X}}_{1+s-(1+s-d),s} = \overline{\mathfrak{X}}_{d,s}$.
Since $M$ is a brick, $\inf(M)$ is a brick in $\modd \cK_s$. Recall that $\Hom(X_U,M) = 0$ for all $U \in \Gr_{d,r}$ implies that, viewing $M$ as a representation, the linear map $M(\gamma_1) \colon M_1 \to M_2$ corresponding to $\gamma_1$ is injective. Since the map is not affected by $\inf$, $\inf(M)(\gamma_1) \colon M_1 \to M_2$ is also injective. Therefore we conclude for the first basis vector $e_1 \in k^s$ and $V:= \langle e_1 \rangle_k$ that $0 = \Hom(X_V,\inf(M))$. By $\ref{RegularAndQuasi}$ we find $W \in \Gr_{1,s}$ with $0 \neq \Hom(X_W,\inf(M))$.
\end{proof}

\subsection{Numerous components lying in $\CSR_d$.}\label{SectionConstant}

\noindent In this section we use the Simplification methdod to contruct a family of regular components, such that every vertex in such a regular component corresponds to a module in $\CSR_d$. By the next result it follows that $\cX \subseteq \overline{\mathfrak{X}}_{d,r}$ implies $\cE(\cX) \subseteq \overline{\mathfrak{X}}_{d,r}$.

\begin{Lemma}\cite[1.9]{Ker2}\label{Filtration}
Let $X,Y$ be modules with $\Hom(X,Y)$ non-zero. If $X$ and $Y$ have filtrations $X = X_0 \supset X_1 \supset \cdots \supset X_r \supset X_{r+1} = 0, Y = Y_0 \supset Y_1 \supset \cdots \supset Y_s \supset Y_{s+1} = 0$, then there are $i,j$ with $\Hom(X_i/X_{i+1},Y_j/Y_{j+1}) \neq 0$.
\end{Lemma}

\noindent For a regular module $M \in \cK_r$ denote by $\cC_M$ the regular component that contains $M$.

\begin{proposition}\label{NumberComponents}
Let $1 \leq d < r$ and $\cX$ be a family of pairwise orthogonal bricks in $\overline{\mathfrak{X}}_{d,r}$. Then
\[\varphi \colon \ind \cE(\cX) \to \cR;  M \mapsto \cC_M\]
is an injective map such that for each component $\cC$ in $\im \varphi$ we have $\cC \subseteq \CSR_d$. Here $\ind \cE(\cX)$ denotes the category of a chosen set of representatives of non-isomorphic indecomposable objects of $\modd \cK_r$ in $\cE$.
\end{proposition}
\begin{proof}
Since each module in $\overline{\mathfrak{X}}_{d,r}$ is regular, $\ref{Simplification}$ implies that every module $N \in \ind \cE(\cX)$ is contained in a regular component $\cC_M$ and is quasi-simple. By $\ref{Filtration}$ the module $N$ satisfies $\Hom(X_U,N) = 0 = \Hom(N,X_U)$ for all $U \in \Gr_{d,r}$. But now $\ref{Finalall}$ implies that every module in $\cC_M$ has constant $d$-socle rank. The injectivity of $\varphi$ follows immediatly from $\ref{Simplification}$.
\end{proof}

\begin{corollary}\label{alotofcomponents}
There exists an infinite set $\Omega$ of regular components such that for all $\cC \in \Omega$ 
\begin{enumerate}[topsep=0em, itemsep= -0em, parsep = 0 em, label=$(\alph*)$]
\item $\cW(\cC) = 0$, in particular every module in $\cC$ has constant rank and
\item $\cC$ does not contain any bricks.
\end{enumerate}
\end{corollary}
\begin{proof}
Let $\cC$ be a regular component that contains a brick and $\cW(\cC) = 0$. Let $M := M_\cC$, then $M \in \overline{\mathfrak{X}}_{1,r}$. Apply $\ref{NumberComponents}$  with $\cX = \{M\}$ and set $\Omega := \im \varphi \setminus \{ \cC_M\}$. Let $N \in \cE(\cX) \setminus \{M\}$ be indecomposable. $N$ is quasi-simple in $\cC_N$ and has a $\{M\}$-filtration $0 = N_0 \subset \ldots \subset N_l = N$ with $l \geq 2$ and $N_1 = M = N_l/N_{l-1}$. Hence 
\[ N \to N_l / N_{l-1} \to N_1 \to N\]
is a non-zero homomorphism that is not injective. Therefore $N$ is not a brick. This finishes the proof, since every regular brick in $\modd \cK_r$ is quasi-simple \cite[9.2]{Ker3} and $\End(\tau^l N) \cong \End(N) \neq k$ for all $l \in \ZZ$.
\end{proof}

\noindent Now we apply our results on the Simplification method to modules $E_d$ constructed in $\ref{NOstable}$.

\begin{Definition}\cite[3.6]{CFP1}
Denote with $\GL_{r}$ the group of invertible $r \times r$-matrices which acts on $\bigoplus^r_{i=1} k\gamma_i$ via $g.\gamma_j = \sum^r_{i=1} g_{ij} \gamma_i$ for $1 \leq j \leq r$, $g \in \GL_{r}$. For $g \in \GL_r$, let $\varphi_g \colon \cK_r \to \cK_r$ the algebra homomorphism with $g.e_1 = e_1$, $g.e_2 = e_2$ and $\varphi_g(\gamma_i) = g.\gamma_i$, $1 \leq i \leq r$. For a $\cK_r$-module $M$ denote the pullback of $M$ along $\varphi_g$ by $M^{(g)}$. The module $M$ is called $\GL_r$-stable if $M^{(g)} \cong M$ for all $g \in \GL_r$.
\end{Definition}

\begin{Theorem}
Let $2 \leq d < r$, then there exists a wild full subcategory $\cE \subseteq \modd \cK_r$ and an injection
\[ \varphi_d \colon \ind \cE \to \cR; M \mapsto \cC_M ,\]
such that for each component $\cC$ in $\im \varphi_d$ we have $\cC \subseteq \CSR_d$ and no module in $\cC$ is $\GL_r$-stable. 
\end{Theorem}
\begin{proof}
Fix $2 \leq d < r$ and let $E_d$ as in $\ref{NOstable}$ with $V,W \in \Gr_{1,r}$ and $\Hom(X_V,E_d) = 0 \neq \Hom(X_W,E_d)$. Set $\cX :=\{E_d\}$ and let $M \in \cE(\cX)$. By $\ref{NumberComponents}$ we get an injective map 
\[ \varphi_d \colon \ind \cE(\cX) \to \cR; M \to \cC_M\]
such that each component $\cC$ in $\im \varphi_d$ satisfies $\cC \subseteq \CSR_d$. \\
Moreover $\cE(\cX)$ is a wild full subcategory of $\modd \cK_r$ by $\ref{Simplification}$. Let $M \in \cE(\cX)$ be indecomposable. Then $M$ has a filtration $0 = Y_0 \subset Y_1 \subset \cdots \subset Y_m$ with $Y_l/Y_{l-1} = E_d$ for all $1 \leq l \leq m$. By $\ref{Filtration}$ we have $0 = \Hom(X_U,M)$ and since $E_d = Y_1 \subseteq  M$, we conclude $0 \neq \Hom(X_W,M)$. This proves that $M$ does not have constant $1$-socle rank. Therefore $\cC_M$ contains a module that is not of constant $1$-socle rank. 
By \cite[3.6]{CFP1} the module $M$ is not $\GL_r$-stable and by \cite[2.2]{Far1} no module in the component is $\GL_r$-stable. 
\end{proof}

\subsection{Components lying almost completely in $\CSR_d$}

\noindent The following Definition and two Lemma are a generalization of \cite[4.7, 4.13]{Wor3} and \cite[3.7]{Wor1}. We omit the proofs.

\begin{Definition}
Let $M$ be an indecomposable $\cK_r$-module, $1 \leq d < r$ and $U \in \Gr_{d,r}$. $M$ is called $U$-trivial if $\dim_k \Hom(X_U,M) = \dim_k M_1$. 
\end{Definition}
Note that the sequence 
$0 \to P^{r-d}_1 \to P_2 \to X_U \to 0$ and left-exactness of $\Hom(-,M)$ imply that $\dim_k \Hom(X_U,M) \leq \dim_k M_1$. 

\begin{Lemma}\label{Utrivial}
Let $M$ be a regular $U$-trivial module. Then either $\Ext(X_V,\tau M) = 0 = \Hom(X_V,\tau^{-1} M)$ for all $V \in \Gr_{d,r}$ or $M$ is isomorphic to $X_U$ or $\tau X_U$.
\end{Lemma}

\begin{Lemma}\label{lLemma}
Let $M$ be regular quasi-simple in a regular component $\cC$ such that $\Ext(X_U,\tau M) = 0 = \Hom(X_U,\tau^{-1} M)$ for all $U \in \Gr_{d,r}$. If $M$ does not have constant $d$-socle rank, then a module $X$ in $\cC$ has constant $d$-socle rank if and only if $X$ is in $(\to \tau M) \cup (\tau^{-1} M \to)$.
\end{Lemma}

\begin{corollary}\label{corlift}
Let $3 \leq r < s$, $d:=1+s-r$ and $1 \leq l \leq s -r$. Let $M$ be an indecomposable $\cK_r$-module in $\overline{\mathfrak{X}}_{d,r}$ that is not elementary. Denote by $\cC$ the regular component that contains $\inf(M)$. 
\begin{enumerate}[topsep=0em, itemsep= -0em, parsep = 0 em, label=$(\alph*)$]
\item Every module in $\cC$ has constant $d$-socle rank and
\item $N \in \cC$ has constant $l$-socle rank if and only if $N \in (\to \tau \inf(M)) \cup (\tau^{-1} \inf(M) \to)$.
\end{enumerate}
\end{corollary}

\begin{proof}
$(a)$ is an immediate consequence of $\ref{goingup}$.\\
$(b)$ Consider the indecomposable projective module $P_2 = \cK_r e_1$ in $\modd \cK_r$. We get 
\[ \Hom(\inf(P_2),\inf(M)) \cong \Hom(P_2,M) = \dim_k M_1 = \dim_k \inf(M)_1 =: q.\]
Since $\dimu \inf(P_2) = (1,r) = (1,s-(s-r))$ we find $W \in \Gr_{s-r,s}$ with $\inf(P_2) = X_W$. Now let $1 \leq l \leq s-r$. By  $\ref{dim}$ there is $U \in \Gr_{l,s}$ and an epimorphism $\pi \colon X_U \to X_W$. 
Let $\{f_1,\ldots,f_q \}$ be a basis of $\Hom(\inf(P_2),\inf(M))$. Since $\pi$ is surjective the set $\{f_1 \pi,\ldots,f_q \pi\} \subseteq \Hom(X_U,\inf(M))$ is linearly independent. Hence $q \leq \dim_k \Hom(X_U,\inf(M)) \leq \dim_k \inf(M)_1 = q$ holds and $\inf(M)$ is $U$-trivial. \\
Since $M$ is not elementary, $\inf(M)$ is not elementary and therefore not isomorphic to $X_U$ or $\tau X_U$. 
Now $\ref{Utrivial}$ yields that 
$ \Ext(X_W,\tau \inf(M)) = 0 = \Hom(X_W,\tau^{-1} \inf(M)) \  \text{for all} \ W \in \Gr_{l,s}$.
By $\ref{RegularAndQuasi}$ the module $\inf(M)$ does not have the constant $l$-socle rank for $1 \leq l \leq s-r$. Note that $M$ is in regular and therefore $\inf(M)$ is a quasi-simple module. Now apply Lemma $\ref{lLemma}$.
\end{proof}

\begin{example}
Let $r \geq 3$ and $\cC$ be a regular component with $\cW(\cC) = 0$ such that $M_\cC$ is not a brick $($see $\ref{alotofcomponents})$ and in particular not elementary. Then $M_\cC \in \overline{\mathfrak{X}}_{1,r}$ and we can apply $\ref{corlift}$. Figure $\ref{Fig:RCGM}$ shows the regular component $\cD$ of $\cK_s$ containing $\inf(M_\cC)$. Every module in $\cD$ has constant $d:= 1 + s-r$ socle rank. But for $1 \leq q \leq r-s$, a module in this component has constant $q$-socle rank if and only if it lies in the shaded region.

\begin{figure*}[!h]
\centering 

\includegraphics[width=0.6\textwidth, height=150px]{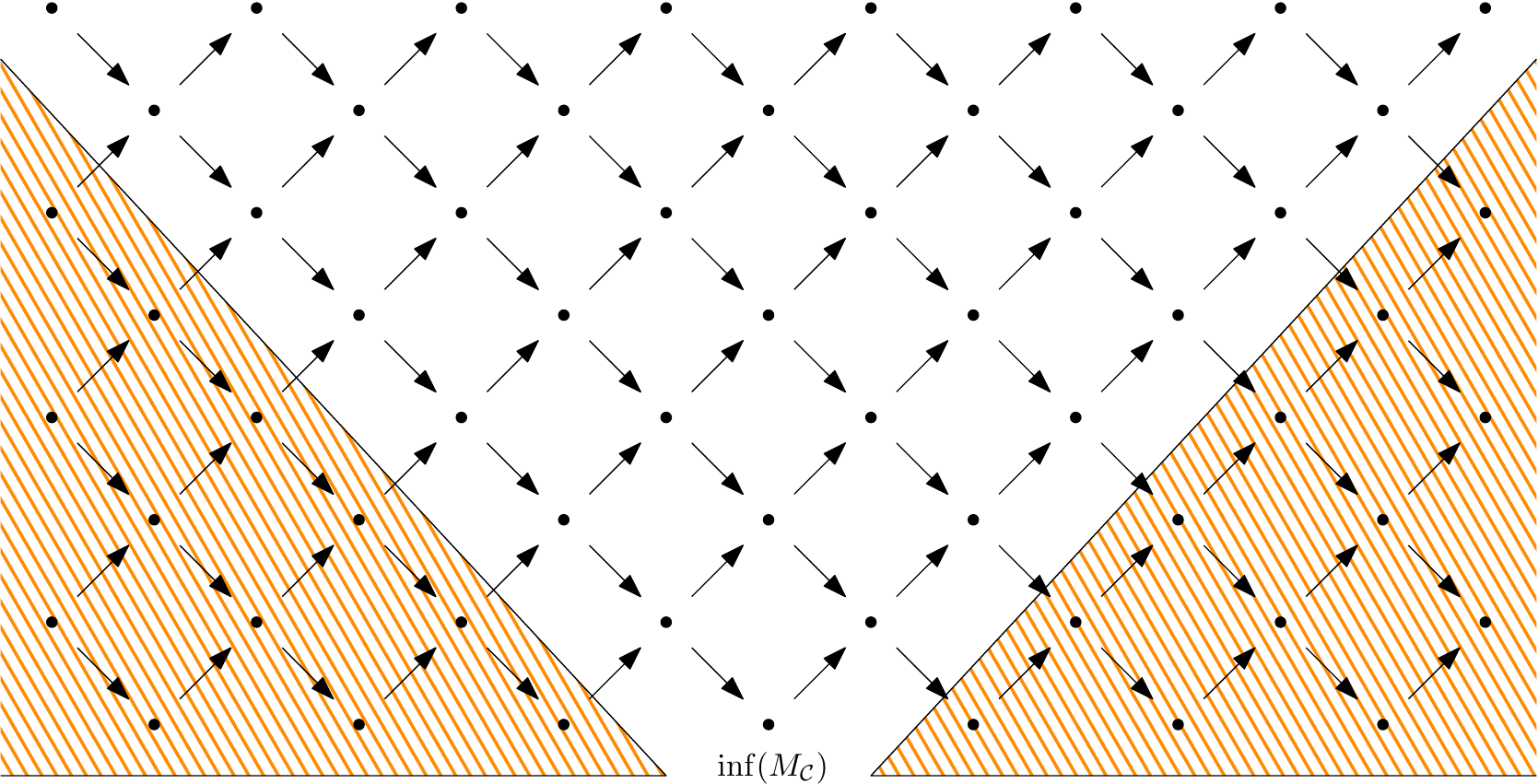}
\caption{Regular component $\cD$ containing $\inf(M_{\cC})$.}
\label{Fig:RCGM}
\end{figure*}
\end{example}

\section{Wild representation type}

\subsection{Wildness of strata}\label{WildnessSection}

As another application of the Simplification method and the inflation functor $\inf^s_s \colon \modd \cK_r \to \modd \cK_s$ we get the following result.

\begin{Theorem}
Let $s \geq 3$ and $1 \leq d \leq s-1$. Then $\Delta_d = \ESP_{d} \setminus \ESP_{d-1} \subseteq \modd \cK_s$ is a wild subcategory, where $\ESP_{0} := \emptyset$.
\end{Theorem}
\begin{proof}
For $d = 1$ consider a regular component $\cC$ for $\cK_s$ that contains a brick $F$. By $\ref{HomTau}$ we find a module $E$ in the $\tau$-orbit of $F$ that is in $\ESP_1$ and set $\cX:=\{E\}$. Then $E$ is brick since $\End(E,E) \cong \End(F,F) = k$ and $\dim_k \Ext(E,E) \geq 2$ by $\ref{Extgeq2}$. Therefore $\cE(\cX)$ is wild category $($see $\ref{Simplification})$. As $\ESP_1$ is closed under extensions, it follows $\cE(\cX) \subseteq \ESP_1$. Note that this case does only require the application of $\ref{Simplification}$.\\
Now let $d > 1$ and $r:= s-d+1 \geq 2$. Consider the projective indecomposable $\cK_r$-module $P:= P_2$ with $\dimu P = (1,r)$. By $\ref{RegularAndQuasi}$ $\inf(\cP)$ is a regular quasi-simple module  in $\modd \cK_s$ with $\dim_k \Ext(\inf(P),\inf(P)) \geq 2$. Since $P$ is in $\ESP_1$, we have $0 = \Hom(X_U,P)$ for all $U \in \Gr_{1,r}$. Hence $\ref{goingup}$ implies $0 = \Hom(X_U,\inf(\cP)) $ for all $U \in \Gr_{1+s-r} = \Gr_{d,s}$, so that $\inf(\cP)$ is in $\ESP_{d}$.\\
Let $\cX := \{\inf(\cP)\}$, then $\cE(\cX)$ is a wild category and since $\ESP_d$ is extension closed it follows $\cE(\cX) \subseteq \ESP_d$. Since  $\dimu \inf(P) = (1,s-d+1)$ we find $V \in \Gr_{d-1,s}$ $($see $\ref{dim})$ with $\inf(P) = X_V$ and $0 \neq \End(\inf(P)) = \Hom(X_V,\inf(P))$. That means $\inf(\cP) \notin \ESP_{d-1}$. Since $\ESP_{d-1}$ is closed under submodules we have $\cE(\cX) \cap \ESP_{d-1} = \emptyset$. Hence $\cE(\cX) \subseteq \ESP_d \setminus \ESP_{d-1}$. 
\end{proof}

\begin{Remarks} 
\begin{enumerate}
\item Note that all indecomposable modules in the wild category $\cE(\cX)$ are quasi-simple in $\modd \cK_s$ and $\cE(\cX) \subseteq \ESP_d \setminus \ESP_{d-1}$ are quasi-simple in $\modd \cK_s$.
\item For $1 \leq d < r$ we define $\EKP_d: = \{M \in \modd \cK_r \mid \forall U \in \Gr_{d,r} \colon \Hom(D\tau X_U,M) = 0\}$.
One can show that $M \in \EKP_d$ if and only if ${y^M_T}_{| M^d_1} \colon M^d_1 \to M$ is injective for all linearly independent tuples $T$ in $(k^r)^d$. From the definitions we get a chain of proper inclusions
$\ESP_{r-1} \supset \ESP_{r-2} \supset \cdots \ESP_{1} = \EKP_{1} \supset \EKP_2 \supset \cdots \supset \EKP_{r-1}$.
By adapting the preceding proof it follows that $\EKP_{r-1}$ is wild. Moreover it can be shown that for each regular component $\cC$ the set $(\EKP_{1} \setminus \EKP_{r-1}) \cap \cC$ is empty or forms a ray.
\end{enumerate}
\end{Remarks}

\begin{proposition}
Let $r = 2$, $B:= B(3,r)$ the Beilinson algebra and $\EKP(3,2)$ the full subcategory of modules in $\modd B(3,2)$ with the equal kernels property $($see \cite[2.1,3.12]{Wor1}$)$. The category $\EKP(3,2)$ is of wild representation type.
\end{proposition}
\begin{proof}
Consider the path algebra of the extended Kronecker quiver $Q = 1 \rightarrow 2\rightrightarrows 3$. Since the underlying graph of $Q$ is not a Dynkin or Euclidean diagram, the algebra $A$ is of wild representation type by \cite{Do1}. It is known that there exists a preprojective tilting module $T$ in $\modd A$ with $\End(T) \cong B(3,2)$, see for example \cite[]{Ung1} or \cite[4.]{Wor2}. We sketch the construction. The start of the preprojective component of $A$ is illustrated in Figure $\ref{Fig:RCTM}$ and the direct summands of $T$ are marked with a square. 

\begin{figure*}[!h]
\centering 
\includegraphics[width=0.5\textwidth, height=75px]{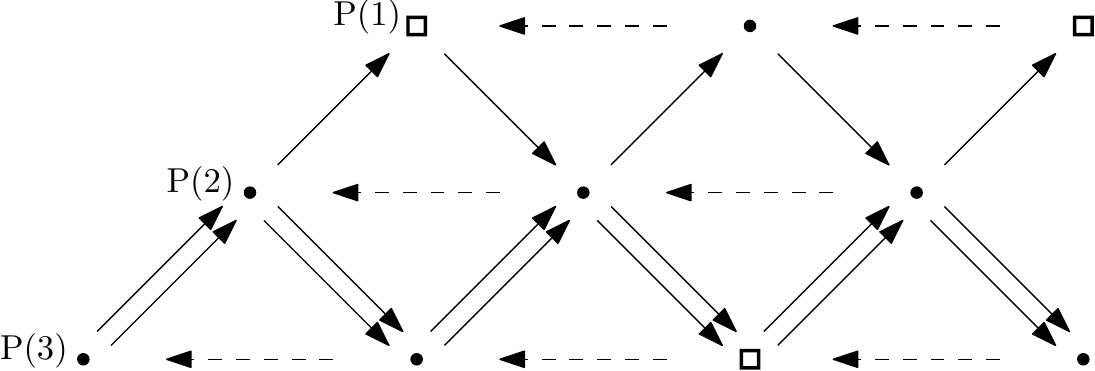}
\caption{Start of the preprojective component of $A$.}
\label{Fig:RCTM}
\end{figure*} One can check that $T = P(1) \oplus \tau^{-2} P(1) \oplus \tau^{-2} P(3)$ is a tilting module. Since preprojective components are standard \cite[2.4.11]{Ri5}, one can show that $\End(T)$ is given by the quiver in Figure $\ref{Fig:Bei}$, bound by the relation $\alpha_2 \alpha_1 + \beta_2 \beta_1$. Moreover it follows from the description as a quiver with relations that $\End(T) \cong B(3,2)$.

\begin{figure}[!h]
\centering 
\tikzstyle{every node}=[]
\tikzstyle{edge from child}=[]
\begin{tikzpicture}[->,>=stealth',auto,node distance=3cm,
  thick,main node/.style={circle,draw,font=\sffamily\Large\bfseries}]

  \node (1) {$\bullet$};
  \node (2) [right of=1] {$\bullet$};
  \node (3) [right of=2] {$\bullet$};
   \node[color=black] at (1.4,0.7) {$\alpha_{1}$};
   \node[color=black] at (1.4,-0.8) {$\beta_{1}$};
      \node[color=black] at (4.4,0.7) {$\alpha_{2}$};
   \node[color=black] at (4.4,-0.8) {$\beta_{2}$};
   \node[color=black] at (0,-0.3) {$1$};
   \node[color=black] at (3,-0.3) {$2$};
   \node[color=black] at (6,-0.3) {$3$};
  \path[every node/.style={font=\sffamily\small}]
    (1) edge[bend right] node [left] {} (2)
    (1) edge[bend left] node [left] {} (2)
    (2) edge[bend right] node [left] {} (3)
    (2) edge[bend left] node [left] {} (3)
    ;
\end{tikzpicture}
\caption{Ordinary quiver of $\End(T)$.}
\label{Fig:Bei}
\end{figure}
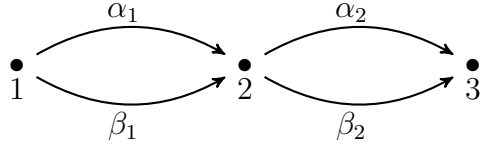
 Since $A$ is hereditary, the algebra $B(3,2)$ is a concealed algebra \cite[VIII 4.6]{Assem1}. By \cite[XVIII 5.0,5.1]{Assem3} the functor $\Hom(T,-) \colon \modd A \to \modd B$ induces an equivalence $G$ between the regular categories $\add \cR(A)$ and $\add \cR(B)$ and we have an isomorphism between the two Grothendieck groups $f \colon K_0(A) \to K_0(B)$ with $\dimu G(M) = \dimu \Hom(T,M) = f(\dimu M)$ for all $M \in \modd A$. Now we will make use of the homological characterisation \cite[2.5]{Wor1} \[\EKP(3,2) = \{ M \in \modd B(3,2) \mid \forall \alpha \in k^r \setminus \{0\}: \Hom(X^1_\alpha,M) = 0 \}.\]
 At first note that each $X^1_\alpha$ is regular; we have $\dim_k \Hom(X^1_\alpha,X^1_\alpha) \neq 0$ and $0 \neq \Hom(X^1_\alpha,X^1_\alpha) = \Ext(X^1_\alpha,\tau X^1_\alpha)$  \cite[3.15]{Wor3}, therefore $X^1_\alpha$ is not in $\EIP(3,2) \cup \EKP(3,2)$ and by \cite[2.7]{Wor1} regular. Moreover $\dimu X^1_\alpha$ is independent of $\alpha$. Hence we find for each $\beta \in k^r\setminus \{0\}$ a regular indecomposable module $U_\beta$ in $\modd A$ with $G(U_\beta) = X^1_\beta$ and $\dimu U_\beta = \dimu U_\alpha$ for all $\alpha \in k^r\setminus \{0\}$. Now let $M$ be in $\modd A$ a regular brick with $\dim_k \Ext(M,M) \geq 2$ $($see \cite[5.1]{KerLuk2}$)$. By  the dual version of \cite[4.6]{Ker3} we find $l \in \NN_0$ with $\Hom(U_\alpha,\tau^{-l} M) = 0$ for all $\alpha \in k^r\setminus \{0\}$. Set $N := \tau^{-l} M$ and $\cX := \{N\}$. $N$ is a regular brick with $\dim_k \Ext(N,N) = \dim_k \Ext(M,M) \geq 2$ and therefore $\cE(\cX)$ is a wild category in $\add \cR(A)$ \cite[1.4]{KerLuk2}. By $\ref{Filtration}$ we have $\Hom(U_\alpha,L) = 0$ for all $L$ in $\cE(\cX)$ and all $\alpha \in k^r\setminus \{0\}$. Hence $0 = \Hom(U_\alpha,L) = \Hom(G(U_\alpha),G(L)) = \Hom(X^1_\alpha,L)$ for all $\alpha \in k^r\setminus \{0\}$. This shows that the essential image of $G$ is a wild subcategory contained in $\EKP(3,2)$.
\end{proof}

\subsection{The module category of $E_r$}

Throughout this section we assume that $\Char(k) = p > 0$ and $r \geq 2$. Moreover let $E_r$ be a $p$-elementary abelian group of rank $r$ with generating set $\{g_1,\ldots,g_r\}$. For $x_i := g_i - 1$ we get an isomorphism $kE_r \cong k[X_1,\ldots,X_r]/(X^p_1,\ldots,X^p_r)$ of $k$-algebras by sending $X_i$ to $ x_i$ for all $i$.
We recall the definition of the functor $\fF \colon \modd \cK_r \to \modd kE_r$ introduced in \cite{Wor1}. Given a module $M$, $\fF(M)$ is by definition the vector space $M$ and $x_i . m := \gamma_i.m = \gamma_i .m_1 + \gamma_i.m_2 = \gamma_i.m_1$ where $m_i = e_i.m$. Moreover $\fF$ lets the morphisms unchanged, i.e. $\fF(f) \colon \fF(M) \to \fF(N); \fF(f)(m) = f(m)$ for all $f \colon M \to N$.

\begin{Definition} \cite[2.1]{CFP1}
Let $\mathbb{V} := \langle x_1,\ldots,x_r \rangle_k \subseteq \rad(kE_r)$. For $U$ in $\Gr_{d,\mathbb{V}}$ with basis $u_1,\ldots,u_d$ and a $kE_r$-module $M$ we set
\[\Rad_U(M) := \sum_{u \in U} u \cdot M = \sum^d_{i=1} u_i \cdot M, \ \text{and} \]
\[\Soc_U(M) := \{ m \in M \mid  \forall u \in U: u \cdot m = 0\} = \bigcap^d_{i=1} \{ m \in M \mid u_i \cdot m = 0 \}.\]
\end{Definition}

\begin{Definition} \cite[3.1]{CFP1}
Let $M \in \modd kE_r$ and $1 \leq d < r$. 
\begin{enumerate}[topsep=0em, itemsep= -0em, parsep = 0 em, label=$(\alph*)$]	
\item $M$ has constant $d$-$\Rad$ rank $($respectively, $d$-$\Soc$ rank$)$ if the dimension of $\Rad_U(M)$ $($respectively, $\Soc_U(M))$ is independent of the choice of $U \in \Gr_{d,\mathbb{V}}$.
\item $M$ has the equal $d$-$\Rad$ property $($respectively, $d$-$\Soc$ property$)$ if $\Rad_U(M)$ $($respectively, $\Soc_U(M))$ is independent of the choice of $U \in \Gr_{d,\mathbb{V}}$.
\end{enumerate}
\end{Definition}

\begin{proposition}\label{FunctorESP}
Let $M$ be a non-simple indecomposable $\cK_r$-module and $1 \leq d < r$. 
\begin{enumerate}[topsep=0em, itemsep= -0em, parsep = 0 em, label=$(\alph*)$]	
\item $M$ is in $\CSR_d$ if and only if $\fF(M)$ has constant $d$-$\Soc$ rank.
\item $M$ is in $\ESP_d$ if and only if $\fF(M)$ has the equal $d$-$\Soc$ property.
\end{enumerate}
\end{proposition}

\begin{proof}
We fix $\mathbb{V}:= \langle x_1,\ldots,x_r \rangle_k \subseteq \rad(kE_r)$. In the following we denote for $u \in \rad(kE_r)$ with $l(u) \colon M \to M$ the induced linear map on $M$. Let $U \in \Gr_{d,\mathbb{V}}$ with basis $\{u_1,\ldots,u_d\}$, write
$u_j = \sum^r_{i=1} \alpha^i_j x_i$ for all $1 \leq j \leq d$ and 
set $\alpha_j = (\alpha^1_j,\ldots,\alpha^r_j)$. Then $T := (\alpha_1,\ldots,\alpha_d)$ is linearly independent  and $\ker (l(u_j))= \ker(\sum^r_{i=1} \alpha^i_{j} l(x_i)) = \ker  (\sum^r_{i=1} \alpha^i_j \gamma_i) = \ker (x^M_{\alpha_j})$. It follows
\[ \Soc_U(\fF(M)) = \bigcap^d_{i=1} \ker(l(u_i)) =  \bigcap^d_{i=1} \ker(x^M_{\alpha_i}) = \Soc_{\langle T \rangle}(M).\]
Hence $M \in \CSR_d$ implies that $\fF(M)$ has constant $d$-$\Soc$ rank.\\
Now assume that $T = (\alpha_1,\ldots,\alpha_d)$ is linearly independent and set $u_j := \sum^r_{i=1} \alpha^i_j x_i$. Then $U := \langle u_1,\ldots,u_d\rangle \in \Gr_{d,\mathbb{V}}$ and $\Soc_{\langle T \rangle}(T) = \Soc_{U}(\fF(M))$. We have shown that $M$ is in $\CSR_d$ if and only if $\fF(M)$ has constant $d$-$\Soc$ rank. The other equivalence follows in the same fashion.
\end{proof}

\noindent For $1 \leq d < r$ we denote with $\ESP_{2,d}(E_r)$ the category of modules in $\modd kE_r$ of Loewy length $\leq 2$ with the equal $d$-Soc property. As an application of Section $\ref{WildnessSection}$ we get a generalization of \cite[5.6.12]{Be1} and \cite[1]{Bo1}. 

\begin{corollary}
Let $\Char(k) > 0$, $r \geq 3$ and $1 \leq d \leq r-1$. Then $\ESP_{2,d}(E_r) \setminus \ESP_{2,d-1}(E_r)$ has wild representation type.
\end{corollary}
\begin{proof}
Let $1 \leq c < r$. By \cite[2.1.2]{Wor1} and  $\ref{FunctorESP}$, a restriction of $\fF$ to $\ESP_c$ induces a faithful exact functor
\[\fF_{c} \colon \ESP_c \to \modd_2 kE_r\]
that reflects isomorphisms and with essential image $\EIP_{2,c}(E_r)$. Let $\cE \subseteq \ESP_d \setminus \ESP_{d-1}$ be a wild subcategory. Since $\fF_{d-1}$ and $\fF_{d}$ reflect isomorphisms we have $\fF(E) \in \ESP_{2,d}(E_r) \setminus \ESP_{2,d-1}(E_r)$ for all $E \in \cE$. Hence the essential image of $\cE$ under $\fF$ is a wild category.
\end{proof}

\begin{corollary}
Assume that $\Char(k) = p > 2$, then the full subcategory of modules with the equal kernels property in $\modd kE_2$ and Loewy length $\leq 3$ is of wild representation type.
\end{corollary}
\begin{proof}
By \cite[2.3]{Wor1} $(n = 3 \leq p, r=2)$ the functor $\fF_{\EKP(3,2)} \colon \modd B(3,2) \to \modd_3 kE_2$ is a representation embedding with essential image in $\EKP(E_2)$.
\end{proof}


\section*{Acknowledgement}
The results of this article are part of my doctoral thesis, which I am currently
writing at the University of Kiel. I would like to thank my advisor Rolf Farnsteiner for fruitful discussions, his continuous support and helpful comments on an ealier version of this paper.
I also would like to thank the whole research team for the very pleasent working atmosphere and the encouragement throughout my studies.\\
Furthermore, I thank Otto Kerner for answering my questions on hereditary algebras and giving helpful comments, and Claus Michael Ringel for sharing his insights on elementary modules for the Kronecker algebra.


\end{document}